\documentclass[a4paper,12pt,reqno]{article}

\usepackage[utf8]{inputenc}
\usepackage[UKenglish]{babel}
\usepackage{amsmath,amsfonts,amssymb,amsthm}
\usepackage{ulem}
\usepackage{cite}
\usepackage{esint}
\usepackage{color,xcolor}
\usepackage{tikz}
\usepackage{mathrsfs}
\usepackage{graphicx}
\usepackage{subfigure}
\usepackage{enumerate}
\usepackage{pdfsync}
\usepackage{environments}
\numberwithin{equation}{section}
\usepackage[a4paper,margin=2cm]{geometry}
\usetikzlibrary{decorations.pathmorphing,patterns}

\DeclareMathOperator{\tr}{tr}
\DeclareMathOperator{\Cov}{Cov}

\def\E{\mathbb{E}}

\def\d{\,\mathrm{d}}
\def\R{\mathbb{R}}

\def\be{\mathbf{e}}

\def\bq{\mathbf{q}}

\def\bu{\mathbf{u}}
\def\bv{\mathbf{v}}
\def\bw{\mathbf{w}}
\def\bx{\mathbf{x}}

\def\bzer{\mathbf{0}}

\def\bA{\mathbf{A}}
\def\bB{\mathbf{B}}
\def\bC{\mathbf{C}}
\def\bD{\mathbf{D}}
\def\bI{\mathbf{I}}

\def\bP{\mathbf{P}}
\def\bQ{\mathbf{Q}}
\def\bR{\mathbf{R}}
\def\bU{\mathbf{U}}
\def\bX{\mathbf{X}}
\def\bY{\mathbf{Y}}

\def\bSig{\mathbf{\Sigma}}

\def\bOm{\mathbf{\Omega}}

\def\balpha{\boldsymbol{\alpha}}
\def\bxi{\boldsymbol{\xi}}
\def\bzeta{\boldsymbol{\zeta}}

\def\bPhi{\boldsymbol{\Phi}}

\def\dd{\mathrm{d}}

\def\e{\mathrm{e}}

\title{
Dynamical properties of coarse-grained linear SDEs
}
\author{
{Thomas Hudson\footnote{Mathematics Institute, University of Warwick,  T.Hudson.1@warwick.ac.uk}} \and {Xingjie Helen Li\footnote{Department of Mathematics and Statistics, University of North Carolina at Charlotte,  xli47@uncc.edu}}
}
\date{\today}

\begin{document}
\maketitle

\begin{abstract}
    Coarse-graining or model reduction is a term describing a range of approaches used to extend the time-scale of molecular simulations by reducing the number of degrees of freedom. In the context of molecular simulation, standard coarse-graining approaches approximate the potential of mean force and use this to drive an effective Markovian model. To gain insight into this process, the simple case of a quadratic energy is studied in an overdamped setting. A hierarchy of reduced models is derived and analysed, and the merits of these different coarse-graining approaches are discussed. In particular, while standard recipes for model reduction accurately capture static equilibrium statistics, it is shown that dynamical statistics such as the mean-squared displacement display systematic error, even when a system exhibits large time-scale separation. In the linear setting studied, it is demonstrated both analytically and numerically that such models can be augmented in a simple way to better capture dynamical statistics.
\end{abstract}

\medskip


\tableofcontents

\section{Introduction}

\noindent
 The behaviour of a system of interest as it evolves in time is often modelled by a large number of interacting degrees of freedom. While such detailed models provide excellent accuracy in theory, their high resolution often renders the cost of simulating them repeatedly over long time scales prohibitive. This limits the applicability of such models, and as a result, the answers to many scientific questions lie outside the realms of what is computationally feasible. To remedy this, various coarse-graining (or model reduction) strategies have been developed, aiming to create cheap-to-simulate effective models which govern a smaller number of degrees of freedom. Such models significantly expand the possible applications of the computational resources available to theoretical scientists.

Coarse-graining is generally believed to be an effective approach when simulating a dynamical system which exhibits a significant time-scale separation. When present, such separation can be exploited in classical asymptotic approaches such as the method of multiple scales and stochastic averaging \cite{skorokhod2009asymptotic,white2010asymptotic}. These methods tend to assume that it is possible to assess and classify variables \textit{a priori} according to the time-scales on which they evolve. For example, in statistical physics a separation of timescales provides the motivation for the use of stochastic models through the Mori-Zwanzig formalism \cite{Mori1965,Zwanzig1973a,Nordholm1975}, and a range of approximate methods have been developed:
\begin{enumerate}
   \item 
   When the relaxation of the coarse-grained (CG) and unresolved variables exhibit significant timescale separation which is spatially homogeneous, one can approximate complex time correlated effects by a local-in-time fluctuations and dissipation, rendering the effective dynamics Markovian \cite{Chorin2013a,G+21}. There are numerous recent works concerning error estimates and the properties of this Markovian approach; see for example \cite{Legoll2010a,Legoll2012,Legoll2017a,Lelievre2018a,Pasquale2018,Legoll2018b,Hudson_2022a}.
    \item 
In case where there is no significant timescale separation between CG and unresolved variables, a Markovian approach may at best provide satisfactory approximation of equilibrium\cite{Berkowitz1983a,Chorin2000a,Guttenberg2013a,ZhenLi2015a}. A remedy in the case of spatially homogeneous effects is to approximate the memory term in a non-Markovian fashion, and fluctuations are treated as colored noised. Examples of the range of approaches available include \cite{Venturi2017b,Chu_Li2018,Stinis2019,Darve2009a,Yoshimoto2013a,FeiLu2015a,Lei2016a,Lin2021a,ZhenLi2015a}.
\end{enumerate}
The settings where CG strategies would have the most impact are in cases where the dynamical systems are extremely high-dimensional in nature, and so it is usually unclear which variables are slow and which are fast. As such, all CG methods require intuition or theoretical input in determining what an appropriate reduction for a system should be. In the case of molecular dynamics, this problem is widely recognised, and the development of new methods to better select CG variables or reaction coordinates remains a problem of significant interest.

Motivated by this broader problem of effectively coarse-graining models in statistical physics, we seek to advance the rigorous analysis of the accuracy of CG models. Here, we focus on simple case where there is sufficient structure to develop an instructive mathematical theory which can assess the strengths and weaknesses of different CG approaches. Starting from the linear overdamped Langevin equation with a linear coarse-graining map, we derive a hierarchy of approximate coarse-grained models. We then assess the accuracy of these approaches on the basis of their ability to capture observables of interest to molecular dynamics practitioners. The first of these observables is the Gibbs distribution of the reduced model, which captures the static equilibrium properties of the coarse-grained model. The second observable is the equilibrium autocovariance of the coarse-grained evolutions, representing the dynamical statistics of the model.
While the model we study is evidently much simpler than many molecular dynamics models, a particular case where our analysis and the methodology of constructing higher-order approximations in time could already find direct applications includes the normal mode analysis of proteins \cite{li2002coarse, tozzini2005coarse}, where coarse-graining can be performed at the level of residues.

The derivation of the hierarchy of approaches we study is based on a reformulation of the dynamics of the CG variables as an integro-differential equation, similar to the derivation of Generalised Langevin Equations via the Mori-Zwanzig formalism \cite{Mori1965,Zwanzig1961}. The integrals present in the resulting equations are then approximated via the truncation of a series expansion based upon formal asymptotic methods. In this expansion, an approach consistent with other CG strategies naturally appears at leading order, and a Markovian correction similar to that observed in \cite{Hudson_Li2020} appears at first-order.
In particular, the first-order approach can be seen as a form of predictor-corrector strategy in the time evolution, and as we show in our rigorous results, enables us to improve the time approximation error.

After deriving our CG approaches, we proceed to study them both analytically and numerically, focusing in particular on their ability to accurately capture the dynamical statistics for coarse-grained models. This particular focus lies in contrast to estimates obtained in previous rigorous mathematical work \cite{Legoll2010a,Legoll2017a,Lelievre2018a,Legoll2018b}. We note that while pathwise error estimates were presented in \cite{Lelievre2018a,Legoll2018b}, in \cite{Hudson_Li2020} we observed that these estimates are highly accurate for short time predictions, but may not necessarily be accurate over longer timescales. To the best of our knowledge, our results therefore provide the first rigorous analysis of the dynamical statistics for CG approaches. Our analysis proceeds via a direct approach, and we derive comprehensive short and long-time autocovariance error estimates for a two-dimensional case, quantifying the error in terms of the deviation from an optimal choice of CG variables and the extent of timescale separation present in the system. Intriguingly, we find that our first-order approach better captures autocovariance properties of the CG system overall, particular in the case of large time-scale separation, but that the leading-order approach better captures autocovariances over very short time-lags than our first-order approach.

To conclude our study, we explore a series of numerical examples both to validate our analytical results and to establish evidence regarding the validity of our conclusions for higher-dimensional systems. In higher-dimensional examples, we find consistently that our new first-order approach provides greater accuracy for the dynamical properties of the system, while maintaining accuracy on static equilibrium statistics, and we discuss further directions for exploration.

\paragraph{Outline.} The paper is organized as follows: In Section~\ref{sec:Setting}, we define the notation with which we describe our system, and specify what we mean by a coarse-graining map in our setting. In Section~\ref{sec:derivation}, we provide a formal derivation of the approximate dynamics we subsequently study. In Section~\ref{sec:statistics}, we present the statistical properties of the full dynamics and the approximate dynamics, and prove our first main result, Theorem~\ref{th:NDError}, providing a rigorous comparison of the autocovariance functions for the approximate dynamics we propose which is most informative for short time-lags. In Section~\ref{sec:2d_error}, we provide results which are accurate over a wider range of time-lags in a two-dimensional setting. Here, our main results are Theorem~\ref{th:2Derror-longtime} and Theorem~\ref{th:2Derror-shorttime}, which provide estimates where the dependence upon both the timescale separation of the system, and the alignment of a coarse-grained variable with the slowest degree of freedom are made explicit. Finally, in Section~\ref{sec:num}, we complement our rigorous results with a series of numerical examples, confirming that the insights of our rigorous analysis in two dimensions can be used to estimate the errors in high dimensions. For the reader's convenience, detailed proofs of our main analytical results have been left out of the main text, but are provided in Appendix~\ref{app:proofs}. 

\section{Setting}
\label{sec:Setting}
Throughout this paper, we use bold capital letters for matrices, e.g. $\bA$, bold lower-case letters for vectors, e.g. $\bq$, and italics for scalars, e.g. $t$. For matrices, we write the transpose as $\bA^*$. All vectors and numbers used are real.
To quantify the size of matrices, we use the Frobenius norm, which is defined to be
\[
    \|\bA\|_F:=\sqrt{\tr(\bA^*\bA)}.
\]

\subsection{Linear overdamped Langevin equation}
We consider the first-order stochastically-forced linear system of SDEs
\begin{equation}\label{eq:SDE}
    \dd\bq_t = -\bA\bq_t \dd t+\sqrt{2\beta^{-1}}\,\dd\bw_t.
\end{equation}
Since we consider a diffusion independent of position, we note that this system of SDEs can be viewed in either the It\^{o} or Stratonovich sense, and our results hold independently of the precise choice of interpretation.
We assume that $\bA$ is symmetric and strictly positive definite, which entails that the system is ergodic with respect to the invariant measure corresponding to the Boltzmann-Gibbs equilibrium distribution. In this case, the equilibrium distribution takes the form
\[
\mu(\dd\bq) = \frac{1}{Z}\exp\big(-\tfrac12 \beta \bq\cdot(\bA\bq)\big),
\]
i.e. a mean-zero Gaussian distribution $\mathcal{N}(\bzer,\beta^{-1}\bA^{-1})$. We note that while it is possible to consider more general linear systems than \eqref{eq:SDE} for our analysis, we focus on this choice since we are motivated by applications in statistical physics \cite{van1976stochastic}.

The general solution to \eqref{eq:SDE} can be expressed using the matrix exponential and stochastic integral as
\begin{equation}\label{eq:SDE_generalsoln}
  \bq_t = \e^{-\bA t} \bq_0+\int_0^t \e^{\bA(s-t)}\sqrt{2\beta^{-1}} \, \d\bw_s.
\end{equation}
If we suppose that the system is at thermal equilibrium at initial time, so that the initial conditions are distributed according to the equilibrium distribution, $\bq_0\sim \mathcal{N}(\bzer,\beta^{-1}\bA^{-1})$, then it follows that $\bq_t$ is stationary, i.e. it has an identical distribution for any $t$. Moreover, using the treatment of linear SDEs considered in \S3.7 of \cite{Pavliotis}, we can establish that the autocovariance of the solution (irrespective of the initial distribution) is
\begin{equation}\label{eq:fullautocov}
    \E[\bq_t\otimes\bq_s]= \beta^{-1}\bA^{-1}\e^{-|t-s|\bA}+\e^{-t\bA}\Big(\E[\bq_0\otimes\bq_0]-\beta^{-1}\bA^{-1}\Big)\e^{-s\bA}.
\end{equation}
We note that the latter term vanishes if $\bq_0$ is distributed according to the equilibrium distribution.

\subsection{Coarse-grained variables}
\label{sec:CGassumptions}
We suppose that \eqref{eq:SDE} represents an accurate physical system of interest which we would like to simplify for the purposes of generating predictions. In order to do so, we seek to reduce the number of degrees of freedom in the model. In particular, we suppose that there is a linear map $\widetilde{\Phi}\in \mathcal{L}(\R^N;\R^n)$ with $n< N$ which selects a collection of variables of interest. Selecting a linear combination of variables (even for a nonlinear model) is a very standard coarse-graining approach, since it tends to enhance the interpretability of the coarse-grained model.

In a slight abuse of notation, we will identify the linear operator $\widetilde{\Phi}$ with a matrix $\widetilde{\Phi}\in \mathbb{R}^{n\times N}$ defined with respect to fixed orthonormal coordinate bases for $\mathbb{R}^N$ and $\mathbb{R}^n$. We will assume that $\widetilde{\Phi}$ has full rank, which ensures there are no `redundant' coarse-grained variables. In cases where $\widetilde{\Phi}$ is not of full rank, it is always possible to reduce to the case we consider by redefining $\widetilde{\Phi}$ with a smaller $n$ by selecting a maximal collection of linearly independent rows.

Using the matrix $\widetilde{\Phi}$, we denote the coarse-grained variables $\widetilde{\bxi} := \widetilde{\Phi} \bq$. We introduce the
matrix
\begin{equation}\label{matrix_Phi}
\Phi = \big(\widetilde{\Phi}\widetilde{\Phi}^*\big)^{-\frac12}\widetilde{\Phi},
\end{equation}
and corresponding variables $\bxi := \Phi\bq=\big(\widetilde{\Phi}\widetilde{\Phi}^*\big)^{-\frac12}\widetilde{\bxi}$. It is straightforward to verify that
\[
\bP := \Phi^*\Phi = \widetilde{\Phi}^*\big(\widetilde{\Phi}\widetilde{\Phi}^*\big)^{-1}\widetilde{\Phi}
\]
is an orthogonal projection, and so we will assume that we can partition the identity of $\mathbb{R}^{N\times N}$ as $\bI = \bP + \Psi^*\Psi$, where $\Psi\in\R^{m\times N}$ is also a matrix of full rank with orthonormal rows, and $m:=N-n>0$. 

Now defining $\bzeta_t:=\Psi\bq_t$, we can write the original SDE system in block format as
\begin{equation}\label{eq:SDE_block}
  \left(\begin{array}{c}
          \dd\bxi_t\\
          \dd\bzeta_t
        \end{array}\right)=-
      \left(
        \begin{array}{cc}
          \bA_0 & \balpha \\
         \balpha^* & \bA_1
        \end{array}
      \right)
      \left(\begin{array}{c}
          \bxi_t\\
          \bzeta_t
            \end{array}\right)\dd t+\sqrt{2\beta^{-1}}
      \left(\begin{array}{c}
          \dd\bw_{0,t}\\
          \dd\bw_{1,t}
        \end{array}\right),
\end{equation}
    where we have set:
    \begin{equation}\label{eq:SubmatrixDefns}
    \begin{gathered}
        \bA_0 := \Phi\bA\Phi^*\in\R^{n\times n},\quad
        \balpha := \Phi\bA\Psi^*\in\R^{n\times m},\quad
        \bA_1 := \Psi\bA\Psi^*\in\R^{m\times m},\\
        \bw_{0,t} := \Phi\bw_t\in\R^n,\quad
        \text{and}\quad\bw_{1,t} := \Psi\bw_t\in\R^{m}.
    \end{gathered}
    \end{equation}
    
    \subsection{Solution of coarse-grained model}
    Following similar steps to those used to derive the general solution \eqref{eq:SDE_generalsoln} to \eqref{eq:SDE}, we can express the general solution to the equation for $\bzeta_t$ in the block system \eqref{eq:SDE_block} as
    \begin{equation}\label{eq:SDE_block_zetageneralsoln}
        \bzeta_t = \e^{-\bA_1 t}\bzeta_0-\int_0^t \e^{\bA_1(s-t)}\balpha^* \bxi_s \d s+\int_0^t \e^{\bA_1 (s-t)}\sqrt{2\beta^{-1}}\d \bw_{1,s}.
    \end{equation}
    We note that the equation satisfied by $\bxi_t$ in the block system \eqref{eq:SDE_block} is
    \begin{equation}\label{eq:XiFullEqn}
        \dd\bxi_t =- \bA_0\bxi_t\dd t-\balpha \bzeta_t\dd t + \sqrt{2\beta^{-1}}\dd\bw_{0,t}.
    \end{equation}
    By substituting the expression \eqref{eq:SDE_block_zetageneralsoln} into the second term on the right-hand side of this SDE, we have
    \begin{equation}\label{eq:XiCGEqn}
        \begin{split}
        \dd\bxi_t =- \bigg(\bA_0\bxi_t+\balpha \e^{-\bA_1 t}\bzeta_0-\int_0^t \!\!\balpha\e^{\bA_1(s-t)}\balpha^* \bxi_s \d s\bigg)\dd t
        \qquad\qquad\\+\bigg(\int_0^t \!\!\balpha\e^{\bA_1 (s-t)}\sqrt{2\beta^{-1}}\cdot\d \bw_{1,s}\bigg)\dd t + \sqrt{2\beta^{-1}}\dd\bw_{0,t}.
        \end{split}
    \end{equation}
    In this equation, we have obtained a closed equation for $\bxi_t$, up to prescribing an initial condition for $\bzeta_0$.
    
    The derivation we have just performed can be viewed as a form of the Mori-Zwanzig approach to model reduction \cite{Mori1965,Zwanzig1973a}, as we have expressed the evolution of the coarse-grained variables as a self-consistent evolution equation. However, the price for this transformation is that the equation has become non-local in time due to the two integral terms which now appear on the right-hand side of the equation. In the following section, we will approximate these integral terms, to derive coarse-grained models which are local in time.

    \section{Markovian approximate dynamics}
    \label{sec:derivation}
    In this section, we use the integrodifferential equation obtained in \eqref{eq:XiCGEqn} to derive Markovian equations which approximate its solution by replacing the integral terms which require evaluation at time $t$ only. We proceed formally to motivate the rigorous approximation results which follow.
    
    \subsection{Approximation of  integral terms}
    First, we consider the first integral term on the right-hand side of \eqref{eq:XiCGEqn}. By formally expanding $\bxi_s$ about $s=t$, exchanging the order of summation and integration, and rearranging, we obtain:
    \begin{align*}
        I&:=\int_0^t \!\!\balpha\e^{\bA_1(s-t)}\balpha^* \bxi_s \d s\\
        &= \int_0^t \!\!\balpha\e^{\bA_1(s-t)}\balpha^* \bigg(\bxi_t+(s-t)\dd\bxi_t +\sum_{j=2}^{\infty}\frac{(s-t)^j}{j!}\dd^{j}\bxi_t \bigg) \d s.
    \end{align*}
    Truncating the series, we have the approximation
    \[
     I \approx\underbrace{\bigg(\int_0^t \!\!\balpha\e^{\bA_1(s-t)}\balpha^* \dd s \bigg)\bxi_t}_{=:I_1}+\underbrace{\bigg(\int_0^t \!\!\balpha\e^{\bA_1(s-t)}\balpha^*  (s-t) \d s\bigg)\dd\bxi_t}_{=:I_2}.
    \]
    For values of $t$ which are sufficiently large, the integrals $\{I_j\}_{j=1}^{2}$ may be approximated by integrating over the interval $(-\infty,t)$. This leads to the further approximations
    \begin{equation*}
    \begin{aligned}
        I_1&\approx\bigg(\int_{-\infty}^{t}\balpha e^{\bA_1(s-t)}\balpha^*\d s\bigg) \bxi_t=
        \balpha\bA^{-1}\balpha^*\bxi_t,\\ \quad\text{and}\quad
        I_2&\approx\bigg(\int_{-\infty}^{t}\balpha e^{\bA_1(s-t)}\balpha^*(s-t)\d s\bigg)\d \bxi_t= -\balpha\bA_1^{-2}\balpha^*\dd\bxi_t.
        \end{aligned}
    \end{equation*}
    To handle the stochastic integral term on the right-hand side of \eqref{eq:XiCGEqn}, we apply It\^{o}'s formula.
    Recall that for any function $\bu:\R\times\R^k\to\R^l$ which is continuously differentiable in its first argument, and twice continuously differentiable in its second argument, we have
    \begin{equation*}
        \bu(t,\bw_t)-\bu(0,\bzer) 
        = \int_0^t\frac{\partial \bu}{\partial s}(s,\bw_s)\,\dd s 
        + \int_0^t \nabla\bu(s,\bw_s)\cdot\dd \bw_s
        + \frac12\int_0^t \Delta \bu(s,\bw_s)\dd s.
    \end{equation*}
    In this formula, the differential operators $\nabla$ and $\Delta$ are the usual $k$--dimensional gradient and Laplacian, which act component-wise on the vector argument of $\bu$ and the differentiation is performed with respect to the second variable.
    In order to use this result in our case, we set $\bu(s,\bx) = \balpha\e^{\bA_1(s-t)}\bx$ for any $\bx\in\R^m$, which allows us to write
    \begin{equation}\label{eq:ItoApplication}
        \int_0^t \balpha\e^{\bA_1(s-t)}\cdot\dd \bw_{1,s}=\balpha\bw_{1,t}
        -\int_0^t\balpha\bA_1\e^{\bA_1(s-t)}\bw_{1,s}\,\dd s.
    \end{equation}
    We now once again formally Taylor expand $\bw_{1,s}$ about $s=t$ on the right-hand side, and switch the order of summation and integration, and truncate the series, leading to the approximation
    \begin{align*}
        J:&=\int_0^t\balpha\bA_1\e^{\bA_1(s-t)}\bw_{1,s}\,\dd s\\
        &=\int_0^t\balpha\bA_1\e^{\bA_1(s-t)}\Big(\bw_{1,t}+(s-t)\dd\bw_{1,t}+\sum_{j=2}^{\infty} \frac{(s-t)^{j}}{j!}\dd ^{j} \bw_{1,t} \Big)\,\dd s\\
        &\approx\underbrace{\int_0^t\balpha\bA_1\e^{\bA_1(s-t)}\dd s\,\bw_{1,t}}_{=:J_1}
        +\underbrace{\int_0^t\balpha\bA_1\e^{\bA_1(s-t)}(s-t)\dd s\,\dd\bw_{1,t}}_{=:J_2}.
    \end{align*}
        As in the case of $I_1$ and $I_2$, we may approximate the deterministic integral factors by extending the regime of integration to $(-\infty,t)$ and by doing so, we obtain
    \begin{equation*}
        J_1 \approx \balpha \bw_{1,t}\quad\text{and}\quad J_2\approx -\balpha\bA_1^{-1} \dd\bw_{1,t}.
    \end{equation*}
    Note that $J_1$ cancels the first term on the right hand side of \eqref{eq:ItoApplication}.
    
    \subsection{Markovian approximate equations}
    \label{sec:ApproximateEqns}
    The formal arguments made above lead to a series possible approximations to the coarse-grained equation \eqref{eq:XiCGEqn}, depending on our choice of leading order approximation for the integral terms $I$ and $J$. Natural choices are to include zero, one or two terms in each of the approximations. In order to state these alternative effective dynamics in a concise form, we define
    \begin{equation}\label{eq:EffectiveMatrixDefns}
        \bB := \bA_0-\balpha\bA_1^{-1}\balpha^*\quad\text{and}\quad\bC:=(\bI+\balpha\bA_1^{-2}\balpha^*)^{-1},
    \end{equation}
    where we recall the definitions of the submatrices $\balpha$, $\bA_0$ and $\bA_1$ made in \eqref{eq:SubmatrixDefns}. The resulting equations are then:
    \begin{align}
   \text{Approach 0:}&&\dd\bxi_t =& -\Big(\bA_0\bxi_t-\balpha\e^{-\bA_1t}\bzeta_0\Big)\dd t+\sqrt{2\beta^{-1}}\dd\bw_{0,t}\label{Approx_eq0}\\
   \text{Approach 1:}&&    \dd\bxi_t =& -\Big(\bB\bxi_t-\balpha\e^{-\bA_1t}\bzeta_0\Big)\dd t+\sqrt{2\beta^{-1}}\dd\bw_{0,t}\label{Approx_eq1}\\
  \text{Approach 2:}&& \bC^{-1}\dd\bxi_t =& -\Big(\bB\bxi_t-\balpha\e^{-\bA_1t}\bzeta_0\Big)\dd t+\sqrt{2\beta^{-1}}\Big(\dd\bw_{0,t}+\balpha\bA_1^{-1}\dd\bw_{1,t}\Big).\label{Approx_eq2}
    \end{align}
    Approach 0 approximates the integral terms by neglecting them, i.e. by setting $I\approx0$ and $J\approx0$. Approach 1 includes the approximation of $I_1$ and $J_1$, and Approach 2 includes the approximation of $I_1$, $I_2$, $J_1$ and $J_2$.

    Note that in the final case where we introduce $\dd\bw_{1,t}$, we could replace the sum of two {independent} Brownian motions with a single Brownian motion
    \[
        \sqrt{2\beta^{-1}}\Big(\dd\bw_{0,t}+\balpha\bA_1^{-1}\dd\bw_{1,t}\Big)\sim
        \sqrt{2\beta^{-1}}\Big(\bI+\balpha\bA_1^{-2}\balpha^*\Big)^{\frac12}\dd\bw_{0,t}=
        \sqrt{2\beta^{-1}}\bC^{-\frac12}\dd\bw_{0,t},
    \]
    where the RHS has identical statistics to the sum of Brownian motions on the left due to the additive properties of the variance of Gaussian distributions. This means we can replace Approach 2 by
    \begin{equation}
        \text{Approach $2$: }\quad
        \begin{aligned}
        \dd\bxi_t =& -\bC
        \Big(\bB\bxi_t-\balpha\e^{-\bA_1t}\bzeta_0\Big)\dd t+\sqrt{2\beta^{-1}\bC}\dd\bw_{0,t}.
        \end{aligned}
    \label{Approx_eq3'}
    \end{equation}
    Solutions of this SDE are statistically equivalent to \eqref{Approx_eq2}, and we will use this definition from now on.
    
    \begin{remark}\label{rmk:App2}
    An alternative way to derive Approach~2 is to first integrate Approach~1 to get an approximation of $\bxi_t$ and then plug it into \eqref{eq:XiCGEqn}. This type of strategy is widely used when constructing predictor-corrector numerical integrators \cite{Butcher2003}. In fact, we can continue this Predict–Evaluate–Correct (PEC) strategy to further improve the approximation of $\dd \bxi_t$ if the corrector method is convergent.
    \end{remark}

    \begin{remark}\label{rmk:AppN}
    We note that we chose to truncate the formal series approximating $I$ and $J$ at first-order, but we could have retained further terms. Doing so would have resulted in an SDE system to approximate the coarse-grained which was higher-order in time, requiring an increased number of variables to describe properly. This idea is closely related to the approximation approach studied in \cite{kupferman2002long,ottobre2011asymptotic,Chu_Li2019}. 
      \end{remark}
    
    \begin{remark}
    Note that if the dimensionality of the original model $N$ is very large, directly computing and storing $\bA_1^{-1}$ will be computationally infeasible. In such cases, it follows that assembling the matrices $\bB$ and $\bC$ directly may not be possible, and hence they must be approximated in practice (see for example \cite{krieger2015new}). On the other hand, our derivation does elucidate the right structural choice for the coarse-grained dynamics, and this seems a promising way to inform data-driven approaches to fitting coarse-grained models in future.
    \end{remark}

    \section{Statistical properties of dynamics}
    \label{sec:statistics}
    We now consider the various dynamical approaches we have proposed as Markovian approximations to the true coarse-grained dynamics given in \eqref{eq:XiCGEqn}, and derive and compare various statistical properties of these alternatives with those of the true dynamics. In particular, we study the  equilibrium statistics for the very long-time property and the mean-squared displacement and autocovariance for the short and medium-time properties.
    
    \subsection{Equilibrium statistics}
    \label{sec:equilibrium}
    Independently of the initial conditions for the true dynamics, we note that as $t\to\infty$, the distribution of $\bxi_t$ evolving under the original \eqref{eq:XiCGEqn} tends to a Gaussian equilibrium distribution, $\mathcal{N}(\bzer,\beta^{-1}\Phi\bA^{-1}\Phi^*)$, i.e. a multivariate normal distribution with mean $\bzer$ and covariance matrix $\Phi\bA^{-1}\Phi^*$. Using the theory of Schur complements \cite{HJ85}, we have that
    \begin{equation}\label{eq:SchurInverse}
        \Phi\bA^{-1}\Phi^* = (\bA_0-\balpha\bA_1^{-1}\balpha^*)^{-1} = \bB^{-1}.
    \end{equation}
    For $\bxi_t$ evolving under Approaches 1 and 2, it can be checked that the resulting equilibrium distribution for is identical. On the other hand, the equilibrium distribution for Approach 0 has covariance $\bA_0^{-1}$. These facts are encoded in the following proposition:
    
    \begin{proposition}\label{Prop_equilibrium}
      Approaches 0, 1 and 2 all have uniquely-defined equilibrium distributions, which are Gaussian with mean $\bzer$ and covariances $\beta^{-1}\bA_0^{-1}$, $\beta^{-1}\bB^{-1}$ and $\beta^{-1}\bB^{-1}$ respectively. In each case, the dynamics is ergodic with respect to this distribution.
      
      As a consequence, the equilibrium distributions for Approaches 1 and 2 always agree with the true equilibrium distribution of $\bxi_t$ evolving under the full dynamics. The equilibrium distribution for $\bxi_t$ evolving under Approach 0 agrees with the true equilibrium distribution if and only if the $n$ rows of $\Phi$ span a subspace generated by $n$ linearly independent eigenvectors of $\bA$.
    \end{proposition}
    \medskip

    \noindent
    The first part of this result is a direct corollary of Proposition~4.2 in \cite{Pavliotis} applied to each of the equations, so we omit a proof. The latter part of this result is implied by the following lemma, which provides conditions under which the covariance of of the equilibrium distribution for Approach 0 agrees with the true covariance: a detailed proof is provided in Appendix~\ref{sec:EquivalenceProof}.
    
    \begin{lemma}\label{th:EquivalenceLemma}
    Assuming the definitions given in \eqref{eq:SubmatrixDefns} and \eqref{eq:EffectiveMatrixDefns}, the following statements are equivalent:
    \begin{itemize}
        \item $\bA_0^{-1}=\bB^{-1}$;
        \item $\balpha=\Phi\bA\Psi^*=\bzer$; and
        \item The $n$ rows of $\Phi$ span a subspace generated by $n$ linearly independent eigenvectors of $\bA$.
    \end{itemize}
    \end{lemma}\medskip
    
    \noindent
    We note further that in the case where the rows of $\Phi$ span a collection of eigenspaces $\bA$, we have the stronger result that the equations \eqref{Approx_eq0}--\eqref{Approx_eq2} defining Approaches 0--2 are identical, since the result of Lemma~\ref{th:EquivalenceLemma} guarantees that $\bB=\bA_0$, and moreover
    \[
        \bC = (\bI+\balpha\bA_1^{-2}\balpha^*)^{-1} = \bI.
    \]
    In other words, if the row-space of the coarse-graining map $\Phi$ is exactly a span of eigenvectors of the matrix $\bA$, then the coarse-grained system completely decouples from the remaining variables.
    
    In a practical coarse-graining setting, managing to accurately capture a large collection of eigenspaces of the dynamics via the coarse-graining map without doing so intentionally is highly improbable. The fact that the approaches do collapse in this way does however suggest that, under the simple choice of dynamics \eqref{eq:SDE}, designing a coarse-graining map which captures eigenspaces of $\bA$ will lead to accurate effective dynamics, regardless of the sophistication of the approach chosen.

    \subsection{Mean-squared displacement and covariance}
    \label{sec:MSD}
    Above, we have summarised the equilibrium statistics for the different coarse-graining approaches we proposed in \eqref{Approx_eq0}-\eqref{Approx_eq2}. Similar properties have been studied extensively in previous works, and so the focus of our work here is on the dynamical statistical properties of the system.
    We note that Proposition~\ref{Prop_equilibrium} implies that in general, Approach~0 incorrectly predicts the equilibrium statistics of the coarse-grained variables, and for this reason, we will ignore this approach from now on, considering only Approaches 1 and 2.
    
    In many applications in statistical physics, and particularly in the simulation of molecular diffusion in solutes, a natural dynamical statistical quantity of interest is the mean-squared displacement of a particle. Mathematically, this quantity is closely related to the autocovariance of the evolution, as we discuss here. In particular, the mean-squared displacement of the stochastic process $\bxi_t$ is defined to be
    \begin{equation}\label{MSE_def}
      D(t):=\E\big[\|\bxi_t-\bxi_0\|^2\big],
    \end{equation}
    where $\bxi_0$ is a deterministic initial condition for the evolution of $\bxi_t$.
    We note that in general, this definition is dependent upon the choice of $\bxi_0$, but we suppress this dependence in our notation. We note that we can rewrite $D(t)$ as
    \[
      D(t)=\tr\Big(\E\big[(\bxi_t-\bxi_0)\otimes(\bxi_t-\bxi_0) \big]\Big),
    \]
    where $\otimes$ is the usual dyadic product.
    Then, by adding and subtracting $\E[\bxi_t]$ and $\E[\bxi_0]$ and using standard properties of expectations, we obtain
    \begin{align*}
      &\E\big[(\bxi_t-\bxi_0)\otimes(\bxi_t-\bxi_0) \big]\\
      &\qquad=\E\big[(\bxi_t-\E[\bxi_t]+\E[\bxi_t-\bxi_0]+\E[\bxi_0]-\bxi_0)\otimes(\bxi_t-\E[\bxi_t]+\E[\bxi_t-\bxi_0]+\E[\bxi_0]-\bxi_0) \big],\\
           &\qquad=\Cov(\bxi_t,\bxi_t)+\E[\bxi_t-\bxi_0]\otimes\E[\bxi_t-\bxi_0]+\Cov(\bxi_0,\bxi_0)-\Cov(\bxi_0,\bxi_t)-\Cov(\bxi_t,\bxi_0),
    \end{align*}
    where the covariance matrix for two multivariate random variables is
    \[
        \Cov(\bu,\bv) := \E\Big[\big(\bu-\E[\bu]\big)\otimes\big(\bv-\E[\bv]\big)\Big].
    \]
    Taking the trace, we have
    \[      
    D(t)=\big\|\E[\bxi_t-\bxi_0]\big\|^2+\tr[\Cov(\bxi_t-\bxi_0,\bxi_t-\bxi_0)].
    \]
    We hence see that the mean-squared displacement is the sum of two contributions: the squared mean of the displacement and the trace of the covariance matrix of displacement at time $t$ relative to the position at initial time.
    
    In the case where the mean of the initial condition is zero and remains zero for all time, so that $\E[\bxi_0]=\E[\bxi_t]=\E[\bxi_t-\bxi_0]=\bzer$, the mean-squared displacement is completely determined by the displacement covariance matrix, $\Cov(\bxi_t-\bxi_0,\bxi_t-\bxi_0)$.
    Moreover, if we further assume that the initial condition is deterministic with $\bxi_0=\bzer$, then $\Cov(\bxi_0,\bxi_0)$ and $\tr[\Cov(\bxi_0,\bxi_t)]$ are zeros, and hence $D(t)$ is exactly the trace of $\Cov(\bxi_t,\bxi_t)$. As we will see below, different coarse-graining approaches provide a range of approximations to the covariance matrix.
    
    \subsection{Autocovariance of full model}
    In order to compare the statistics of interest, we first compute the autocovariance of $\bxi_t$ evolving under the full model, \eqref{eq:XiCGEqn}, where the initial conditions $\bxi_0=\Phi\bq_0$ and $\bzeta_0=\Psi\bq_0$ are assumed to be fully deterministic. In this case, it is staightforward to show that the mean and covariance of $\bxi_t$ are
    \begin{equation*}
        \E[\bxi_t] = \Phi\e^{-t\bA}\bq_0,\quad\text{and}\quad
        \Cov(\bxi_s,\bxi_t)
    =\beta^{-1}\Phi\bA^{-1}\big(\e^{-|s-t|\bA}-\e^{-(s+t)\bA}\big)\Phi^*.
    \end{equation*}
    If $s=t$, the latter becomes
    $\Cov(\bxi_t,\bxi_t)=\beta^{-1}\Phi\bA^{-1}\big(\bI-\e^{-2t\bA}\big)\Phi^*$, and we note that the autocovariance at initial time is zero, so the mean-squared displacement is
    \begin{align*}
      D(t)&=\tr[\Cov(\bxi_t,\bxi_t)]+\|\E[\bxi_t-\bxi_0]\|^2\\
      &= \beta^{-1}\tr\big[\Phi\bA^{-1}\big(\bI-\e^{-2t\bA}\big)\Phi^*\big]+\|\Phi(\e^{-t\bA}-\bI)\bq_0\|^2.
    \end{align*}
    As noted above, if $\bq_0=\bzer$, the mean-squared displacement is completely determined by the covariance of the evolution, becoming
    \[
        D(t)= \beta^{-1}\tr\big[\Phi\bA^{-1}\big(\bI-\e^{-2t\bA}\big)\Phi^*\big].
    \]
    Although we have assumed a deterministic initial condition here, we note that the dynamics converges to a unique equilibrium regardless of the precise initial distribution, and indeed as $s,t\to\infty$ with $|s-t|$ bounded, the autocovariance in both of the above cases tends towards
  \[
        \Cov(\bxi_s,\bxi_t) \to \beta^{-1}\Phi\bA^{-1}\e^{-|s-t|\bA}\Phi^*.    
  \]
    To this end, we define the equilibrium autocovariance function $\bR:[0,+\infty)\to\R^{n\times n}$ to be
      \begin{equation}\label{covLag_full}
        \bR(\tau) = \beta^{-1}\Phi\bA^{-1}\e^{-\tau\bA}\Phi^*.
      \end{equation}
    
    \subsection{Autocovariance of approximate models}
    We now compute the autocovariance properties of the approximate models in turn. All results are summarised in Table~\ref{tab:summary}.
    
    \paragraph{Approach 1.}
    Here, the expectation satisfies the equation
    \[
        \frac{\dd}{\dd t}\E[\bxi_t] = -\bB\E[\bxi_t]+\balpha\e^{-\bA_1 t}\E[\bzeta_0],
    \]
    which can be integrated to find that
    \[
    \E[\bxi_t] = \e^{-\bB t}\bxi_0+\int_0^t\e^{\bB (s-t)}\balpha\e^{-\bA_1 s}\E[\bzeta_0]\dd s.
    \]
    If $\bxi_0=\bzer$ and $\E[\bzeta_0]=\bzer$, we see that $\E[\bxi_t]=\bzer$ for all time, and we make this assumption to simplify our calculations.
    To compute the autocovariance, we apply the result of Proposition~3.5 in \cite{Pavliotis}, giving
    \[
        \Cov(\bxi_s,\bxi_t) 
        = \beta^{-1}\bB^{-1}\e^{-|t-s|\bB}-\beta^{-1}\bB^{-1}\e^{-(s+t)\bB}.
    \]
    Taking $s=t+\tau$ and letting $t\to\infty$, we define the equilibrium autocovariance for Approach~1 to be $\bR_1:[0,+\infty)\to\R^{n\times n}$, where
    \begin{equation}\label{covLag_App1}
    \bR_1(\tau):=\beta^{-1}\bB^{-1}\e^{-\tau\bB}.
    \end{equation}
    Since $\bxi_0$ is deterministic, $\Cov(\bxi_0,\bxi_0)=\bzer$ and hence the mean-squared displacement is therefore
    \[
    D_1(t):= \beta^{-1}\bB^{-1}\big(\bI-\e^{-2t\bB}\big).
    \]
    
    \paragraph{Approach 2.}
    For Approach~2, similar computations to those performed for Approach~1 yield
    \[
    \E[\bxi_t] = \e^{-\bC\bB t}\bxi_0+\int_0^t\e^{\bC\bB (s-t)}\balpha\e^{-\bA_1 s}\E[\bzeta_0]\dd s.
    \]
    Again, assuming that $\bxi_0=\bzer$ and $\E[\bzeta_0]=\bzer$ entails that $\E[\bxi_t]=\bzer$ for all $t$, and we make this assumption to simplify our calculations.
    To compute the autocovariance, we again apply the result of Proposition~3.5 in \cite{Pavliotis}, giving
    \[
    \Cov(\bxi_s,\bxi_t) = 2\beta^{-1}\int_0^{\min(s,t)}\e^{(\tau-t)\bC\bB}\bC\e^{(\tau-s)\bB\bC}\dd\tau.
    \]
    Using the definition of the matrix exponential, the integrand can be written as
    \[
        \e^{\tau\bC\bB}\bC\e^{\tau\bB\bC} = \bC\e^{2\tau\bB\bC},
    \]
    which allows us to express the integral explicitly as
    \begin{equation*}
    \Cov(\bxi_s,\bxi_t) 
    =\beta^{-1}\bB^{-1}\e^{-|s-t|\bB\bC}-\beta^{-1}\bB^{-1}\e^{-(s+t) \bB\bC}.
    \end{equation*}
    Again, taking $s=t+\tau$ and letting $t\to\infty$, we define the equilibrium autocovariance for Approach~2 to be $\bR_2:$       \begin{equation}\label{covLag_App2}
        \bR_2(\tau):=\beta^{-1}\bB^{-1}\e^{-\tau\bB\bC}.
    \end{equation}
    The mean-squared displacement is then
    \[
    D_2(t):=\beta^{-1}\bB^{-1}\big(\bI-\e^{-2t\bB\bC}\big).
    \]
 
    A summary of different approaches is given in Table~\ref{tab:summary}.
    \begin{table}[htp!]
        \centering
        \begin{tabular}{|c|l|l|}\hline
             & \multicolumn{1}{|c|}{\rule{0pt}{5mm}Mean-squared displacement} & \multicolumn{1}{|c|}{Equilibrium autocovariance} \\[1mm]
             \hline \rule{0pt}{5mm}
             Full dynamics & $D(\tau)=\beta^{-1}\tr[\Phi\bA^{-1}(\bI-\e^{-2\tau\bA})\Phi^*]$ & $\bR(\tau)=\beta^{-1}\Phi\bA^{-1}\e^{-\tau\bA}\Phi^*$ \\[1mm]
             Approach 1 & $D_1(\tau)=\beta^{-1}\tr[\bB^{-1}(\bI-\e^{-2\tau\bB})]$ & $\bR_1(\tau)=\beta^{-1}\bB^{-1}\e^{-\tau\bB}$ \\[1mm]
             Approach 2 & $D_2(\tau)=\beta^{-1}\tr[\bB^{-1}(\bI-\e^{-2\tau\bB\bC})]$ & $\bR_2(\tau)=\beta^{-1}\bB^{-1}\e^{-\tau\bB\bC}$\\[1mm]
             \hline
        \end{tabular}
        \caption{Summary of long-time equilibrium autocovariance functions and mean-squared displacement for different dynamical approaches under the assumption that $\bxi_0=\bzer$ and $\E[\bzeta_0]=\bzer$.
        $\bPhi$ is defined in \eqref{matrix_Phi}, and
        $\bB := \bA_0-\balpha\bA_1^{-1}\balpha^*$ and $\bC:=(\bI+\balpha\bA_1^{-2}\balpha^*)^{-1}$ are defined in \eqref{eq:EffectiveMatrixDefns}.}
        \label{tab:summary}
    \end{table}
    
    \subsection{A matrix form of Jensen's inequality}
    In order to provide global bounds on the autocovariance error we will use a technical result, which is a version of Jensen's inequality for matrices. The form of the result is an adaptation of aspects of the results of \cite{AMS07} or \cite[Theorem 2.1]{hansen2003jensen}.

    In order to state this result, we introduce the L\"{o}wner partial ordering on matrices \cite{L34}, which will be used throughout the remainder of this work. We write $\bA\leq \bB$ if and only if $\bB-\bA$ is positive definite, and for any $x\in\R$, we write $x\leq \bA$ to mean that $x\bI\leq\bA$. In particular, $0\leq \bA$ means that $\bA$ is positive definite. With this notation in place, we may state and prove the following result.
    
    \begin{theorem}\label{th:Jensen}
    Suppose that $f:(0,+\infty)\to\R$ is monotone and convex, and extend the action of this function to real, positive definite matrices via the standard identification
    \[
      f(\bA) = f(\bQ\bD\bQ^*) = \bQ f(\bD)\bQ^*:=\bQ \,\mathrm{diag}\big[f(\bD_{11}),\dots,f(\bD_{NN})\big]\bQ^*,
    \]
    where $\bQ\bD\bQ^*$ is the diagonalisation of $\bA\in \bR^{N\times N}$, so $\bQ$ is an orthogonal matrix and $\bD$ is a diagonal matrix.
    Then, if $\Phi\in \bR^{n\times N}$ satisfies $\Phi\Phi^*=\bI_n$, it holds that
    \[
        f(\Phi\bA\Phi^*)\leq \Phi f(\bA) \Phi^*  
    \]
    for all symmetric positive definite matrices $\bA$. For a more general matrix $\bPhi$, if we further define $\bSig=\sqrt{\Phi\Phi^*}$  then
    \[
    \bSig^* \Big(f\big(\bSig^{-1}\Phi\bA\Phi^*(\bSig^*)^{-1}\big)\Big)\bSig\leq \Phi f(\bA)\Phi^*.
    \]
    
    \end{theorem}
    
    \begin{proof}
    As stated above, this is a direct consequence of results in \cite{AMS07}. In particular, the map
    $\phi:\R^{N\times N}\to\R^{n\times n}$ defined by
    \[
        \phi(\bA) := \Phi\bA\Phi^*
    \]
    is positive (order--preserving) and unital (identity--preserving), and so conclusion (1) of Proposition~5.2 in \cite{AMS07} implies the result. The latter conclusion is a simple consequence of applying the order-preserving mapping $\psi(\bB) := \bSig\bB\bSig$ to both sides of the inequality.
    \end{proof}
    
    \subsection{Global pointwise-in-time error bounds}
    
    We now employ the result above to prove the following global bounds on the difference between the autocovariance matrices. Full proofs of the global bounds which follow are postponed to Appendix~\ref{sec:NDErrorProof}.
    
    \begin{theorem}\label{th:NDError}
    For all {$\tau>0$}, we have that the difference between the true {equilibrium} autocovariance and the equilibrium autocovariance for Approach 1 can be bounded above and below using the L\"owner partial order as follows:
    \[
        0\leq \bR(\tau)-\bR_1(\tau)\leq \tfrac12\beta^{-1}\tau^2(\bA_0-\bB)=\tfrac12\beta^{-1}\tau^2\balpha\bA_1^{-1}\balpha^*.
    \]
    Likewise, for all $\tau>0$, the difference between the true equilibrium autocovariance and the autocovariance
    of Approach 2 can be bounded above and below as follows:
    \[
        \beta^{-1}\tau(\bC-\bI)\leq \bR(\tau)-\bR_2(\tau)\leq \tfrac12\beta^{-1}\tau^2(\bA_0-\bC\bB\bC)+\beta^{-1}\tau(\bC-\bI).
    \]
    As a particular consequence of the latter bound, there exists $\tau^*>0$ such that
    \[
    \beta^{-1}\tau(\bC-\bI)\leq \bR(\tau)-\bR_2(\tau)\leq\beta^{-1}\tau(\bI-\bC)
    \]
    for all $\tau\in[0,\tau^*]$.
    
    \end{theorem}\medskip

    \noindent
    While the bounds provided by this result are global, it is clear that they do not accurately reflect the error for large values of $\tau$, since the autocovariance functions all decay exponentially. Nevertheless, performing Taylor expansions of the autocovariance functions at $\tau=0$ demonstrate that these bounds are highly accurate for $\tau$ small; a short calculation demonstrates that, as $\tau\to0$, we have
    \begin{equation}
        \label{eq:short_time}
    \begin{aligned}
    \beta\bR(\tau)
    &= \bB^{-1}-\tau\bI+\tfrac12\tau^2\bA_0+O(\tau^3),\\
    \beta\bR_1(\tau)
    &= \bB^{-1}-\tau\bI+\tfrac12\tau^2\bB+O(\tau^3)={\bR(\tau)+O(\tau^2)},\\
    \beta\bR_2(\tau)
    &= \bB^{-1}-\tau\bC+\tfrac12\tau^2\bC\bB\bC+O(\tau^3)={\bR(\tau)+O(\tau)}.
    \end{aligned}
    \end{equation}
    We see directly that the comparison of these expansions would give us the same results for sufficiently small $\tau$.
    
    We note that the first bounds provided by Theorem~\ref{th:NDError} demonstrate that the autocovariance of Approach 1 is always an underestimate of the true autocovariance. Connecting this result with the discussion of the mean-squared displacement in Section~\ref{sec:MSD}, we see that coarse-graining via Approach 1 will always predict a faster rate of self-diffusion than the full dynamics. Our result therefore agrees with the widely-observed fact that coarse-grained models are often less stiff than their fine-grained counterparts. While this property may be desirable for certain problems where shortening equilibriation times can lead to accelerated mixing \cite{tuckerman1991molecular,vlachos2005review}, for problems where better estimates of dynamical statistics are required (see for instance \cite{ZhenLi2022}), Approach~1 demonstrates a systematic bias.
    
    On the other hand, inspecting the Taylor expansion for Approach~2 provided in \eqref{eq:short_time}, we see that this approach systematically overestimates the true autocovariance for short lag times. This overestimation does not persist: indeed, we will confirm this both analytically and numerically in subsequent sections, where we will show that while Approach 1 provides an excellent approximation to the autocovariance for short lags $\tau$, Approach 2 provides a better global approximation, particularly when considering the long-time tail behaviour.
    
    \subsection{Progressive coarse-graining}
    \label{sec:multiple_reduction}
    In addition to the bounds obtained in Theorem~\ref{th:NDError}, we now argue that the application of Jensen's inequality in our analysis has an interesting consequence for the comparison of models obtained by progressively coarse-graining.

    \begin{corollary}
    \label{cor:coarsening}
        Suppose that $\bX\in\R^{d\times n}$ and $\Phi\in\R^{n\times N}$ with $d<n$ are coarse-graining maps which satisfy
        \[
            \bX\bX^* = \bI\quad\text{and}\quad\Phi\Phi^*=\bI,
        \]
        and let $\bY=\bX\Phi\in\R^{d\times N}$ be their composition. Suppose further that, applying approximation Approach~1, we have that:
        \begin{itemize}
            \item $\bq_t\in\R^N$ solves the full dynamics,
            \item $\bxi_t\in \R^n$ solves the approximate dynamics derived using the map $\Phi$, and 
            \item $\bx_t\in \R^d$ solves the approximate dynamics derived using the map $\bY$.
        \end{itemize}
          Then we have the following estimates relating the equilibrium covariances:
    \begin{align}
\lim_{t\to\infty} &\Big\|\Cov(\bY\bq_t,\bY\bq_{t+\tau})-\Cov(\bX\bxi_t,\bX\bxi_{t+\tau})\Big\|_F \nonumber \\
&\qquad \leq\lim_{t\to\infty}\Big\|\Cov(\bY\bq_t,\bY\bq_{t+\tau})-\Cov(\bx_t,\bx_{t+\tau})\Big\|_F, \label{Upper_bound_Coarest}\\       \lim_{t\to\infty}&\Big\|\Cov(\bX\bxi_t,\bX\bxi_{t+\tau})-\Cov(\bx_t,\bx_{t+\tau})\Big\|_F\nonumber \\
  & \qquad  \leq    \lim_{t\to\infty}\Big\|\Cov(\bY\bq_t,\bY\bq_{t+\tau})-\Cov(\bx_t,\bx_{t+\tau})\Big\|_F,\label{Coarest_vs_intermediate}
  \end{align}
    where $\|\cdot\|_F$ is the Frobenius norm.
    \end{corollary}\medskip
    
    \noindent
    We may interpret these estimates as saying that the covariance error necessarily increases with the degree of coarse-graining, and differences in autocovariance between coarse and intermediately coarse-grained models provide guaranteed lower bounds on the error between the finest and coarsest models. A full proof of the result is given in Section~\ref{sec:coarseningproof}, and employs Theorem~\ref{th:Jensen}.

    As a concrete example to illustrate the application of the above result, let us consider taking $\Phi$ to be a map onto the first $n$ variables, and $\bX$ a further reduction onto the first $d$ variables with $d<n$, so that
    \[
    \Phi = \left(\begin{array}{cc}
         \bI_n & \bzer_{n\times (N-n)}
    \end{array}
    \right)\quad\text{and}\quad \bX = \left(\begin{array}{cc}
        \bI_d & \bzer_{d\times(n-d)} 
    \end{array}\right).
    \]
    Here, $\bI_n$ and $\bI_d$ are identity submatrices, and $\bzer_{n\times (N-n)}$ and $\bzer_{d\times(n-d)}$ are submatrices of zeros.
    It follows that the composition of these maps, $\bY:=\bX\Phi\in\bR^{d\times N}$ is a `coarser' coarse-graining map satisfying $\bY\bY^* = \bX\bPhi\bPhi^*\bX^* = \bI$. In the particular example case given above, we have that
    \[
    \bY = \left(\begin{array}{cc}
         \bI_d & \bzer_{d\times (N-d)}
    \end{array}
    \right).
    \]
    In this instance, this result compares the autocovariances of the first $d$ variables under the approximate coarse-graining approaches laid out in Section~\ref{sec:derivation}.

    \begin{remark}
    The first estimate in the statement of Corollary~\ref{cor:coarsening} now says that if we consider the difference between the autocovariance of the first $d$ variables in the full model and the first $d$ variables in the $n$-dimensional coarse-grained model with $d<n$, then the error will be smaller than that committed in the $d$-dimensional coarse-grained model.

    The second estimate says that the error in the autocovariance between the first $d$ variables in the $n$-dimensional model and the autocovariance of the $d$-dimensional model provides a lower bound on the autocovariance error to the full model. As such, comparing the autocovariance of a partially coarse-grained model with a further coarse-grained model can be used as an practical assessment of its potential inaccuracy.
    
    We note that we do not prove a similar analytic result for Approach~2. When the time-scale separation (i.e., the spectral gap) within the system is large, the numerical results in Section~6.2 suggest that Approach~2 also possesses this property; but when the spectral gap is not notable, as seen in some examples of Section~6.3, the conclusion of Corollary~\ref{cor:coarsening} will fail for Approach~2.
    
    As a consequence, we will study a two-dimensional system in detail in the next section as the 2D error estimates provide error bounds and insights for higher-dimensional systems.
    \end{remark}
    
    \section{Coarse-graining of two-dimensional systems}
    \label{sec:2d_error}

    We next perform a detailed analytical comparison of both approaches for two-dimensional systems. In particular, throughout this section we will assume $N=2$ and $n=m=1$, so that the coarse graining map reduces a two-dimensional system to a one-dimensional one. In this case, we obtain improved versions of the error estimates established in Theorem~\ref{th:NDError} which hold for longer time-scales. In doing so, we elucidate various features of coarse-graining which we expect to still be present in higher-dimensional cases, which are studied numerically in Section~\ref{sec:num}.
    
    In the two-dimensional case, we may describe all possible coarse-grained systems of the form \eqref{eq:SDE_block} using only two parameters $(\lambda,\, \theta)$.
    Choosing coordinates such that $\bA$ is diagonal, and rescaling such that the lowest eigenvalue of $\bA$ is $1$, we may assume that
    \begin{equation}\label{drift2D}
    \bA = \left(\begin{array}{cc}
        1 & 0 \\
        0 & \lambda
    \end{array}\right),
    \end{equation}
    where $1\leq\lambda$. In general, any coarse-graining map satisfying the assumptions outlined in Section~\ref{sec:CGassumptions} can be expressed in terms of a single angle $\theta\in(-\frac\pi2,\frac\pi2)$ between the coarse-grained variable and the eigenspace of $\bA$ corresponding to the eigenvalue $1$, so that
    \begin{equation}\label{proj2D}
    \Phi := (\cos\theta\;\sin\theta)\quad\text{and}\quad\Psi = (-\sin\theta\;\cos\theta).
    \end{equation}
    In this context, the parameter $\lambda$ therefore provides us with a way to measure the timescale separation in the system, while the angle parameter $\theta$ provides with a way to explore the alignment of the CG variable with the eigenspace corresponding to the slowest timescale. In particular, we will be interested in comparing the dynamical properties of our CG approaches in the asymptotic regime where $\lambda\gg 1$ and $|\theta|\ll 1$, so that timescale separation is large.
    
\subsection{Measuring autocovariance error}
    In order to quantify the accuracy of the approximate approaches proposed in Section~\ref{sec:ApproximateEqns}, we define the following measures of error. To measure the error between different matrix-valued autocovariance functions, we consider both the
    \begin{align}
        \text{Absolute error at time $\tau$}&:=\|\bR_i(\tau)-\bR(\tau)\|_F,\label{Abs_error}\\[2mm]
        \text{Relative error at time $\tau$}&:=\frac{\|\bR_i(\tau)-\bR(\tau)\|_F}{\|\bR(\tau)\|_F},\label{Rel_Abs_error}
    \end{align}
    where $\bR_i$ is the equilibrium autocovariance for Approach~$i$ with $i=1$ or $2$. To measure the accumulation of error over the time interval $(0,\tau)$, we also consider the following measures:
    \begin{align}
            \text{Absolute $L^1(0,\tau)$ mean error}&:=\frac{1}{\tau}\int_0^\tau\big\|\bR_i(t)-\bR(t)\big\|_{F}\,\dd t, \label{L1_error} \\[2mm]
            \text{Relative $L^1(0,\tau)$ mean error}&:=\frac{\int_0^\tau\big\|\bR_i(t)-\bR(t)\big\|_{F}\,\dd t}{\int_0^\tau\big\|\bR(t)\big\|_{F}\,\dd t}.\label{Rel_L1_error}
               \end{align}
           The choice to normalise the absolute $L^1$ error by dividing by the lag time $\tau$ has been made to give the error rates consistent units.
     
    \subsection{Autocovariance error for large time-scale separation}
    Figure~\ref{fig:ACF_LagTime} illustrates the absolute and $L^1(0,\tau)$ mean error in the autocovariance measured in our 2D setting where $\lambda=2$ and $\theta$ varies. We note that the errors for Approach~1 are smaller than those of Approach~2 when the lag time $\tau$ is small, and the opposite is true when $\tau$ becomes large. Further, the errors increase as the coarse-graining map deviates more from the projection onto the eigenspace for the minimal eigenvalue, i.e. when $\theta$ becomes larger.
    
    \begin{figure}
           \centering
          \subfigure[Absolute ACF error]{ \includegraphics[width = 0.45\textwidth]{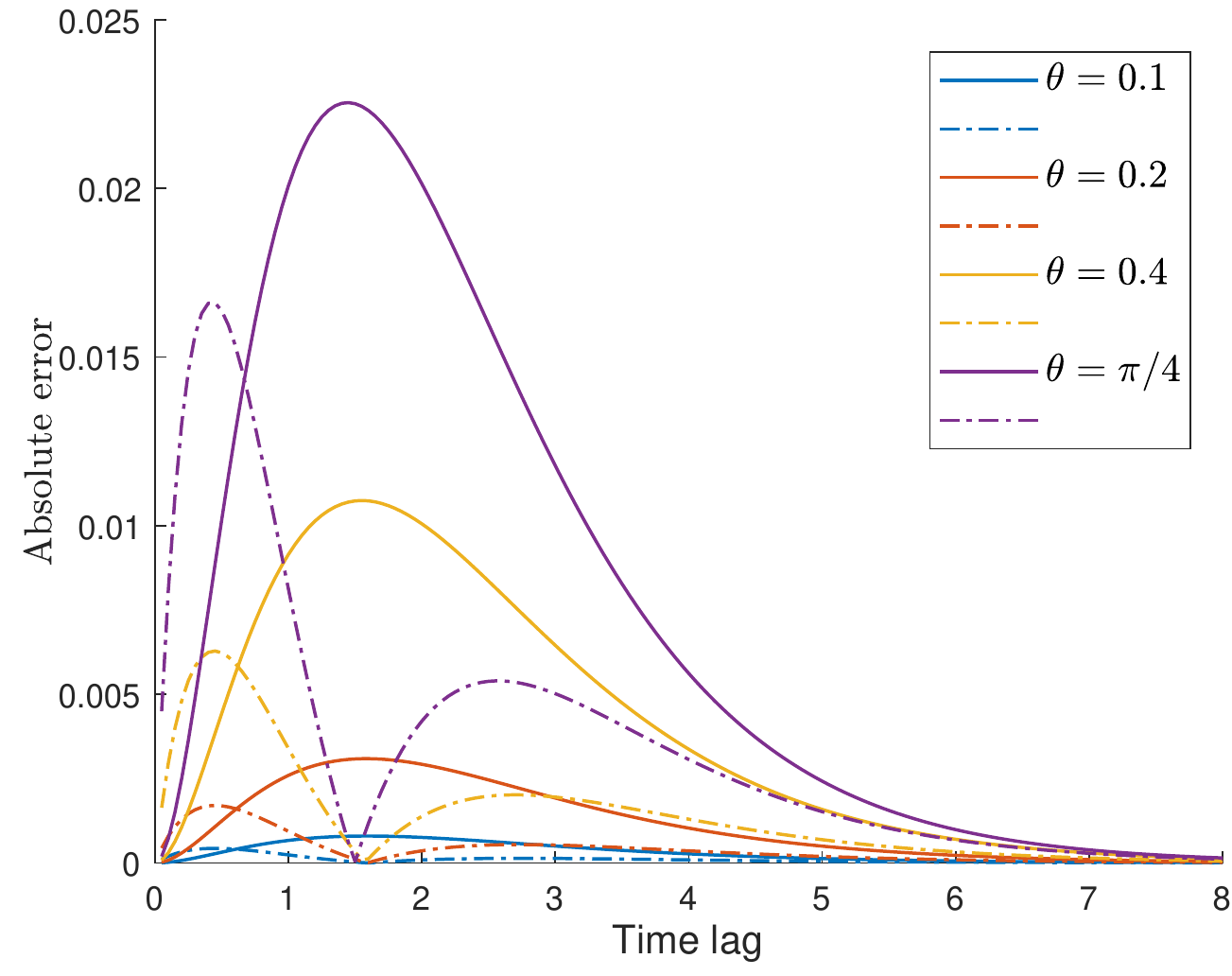}}
           \subfigure[$L^1$ ACF mean error]{ \includegraphics[width = 0.45\textwidth]{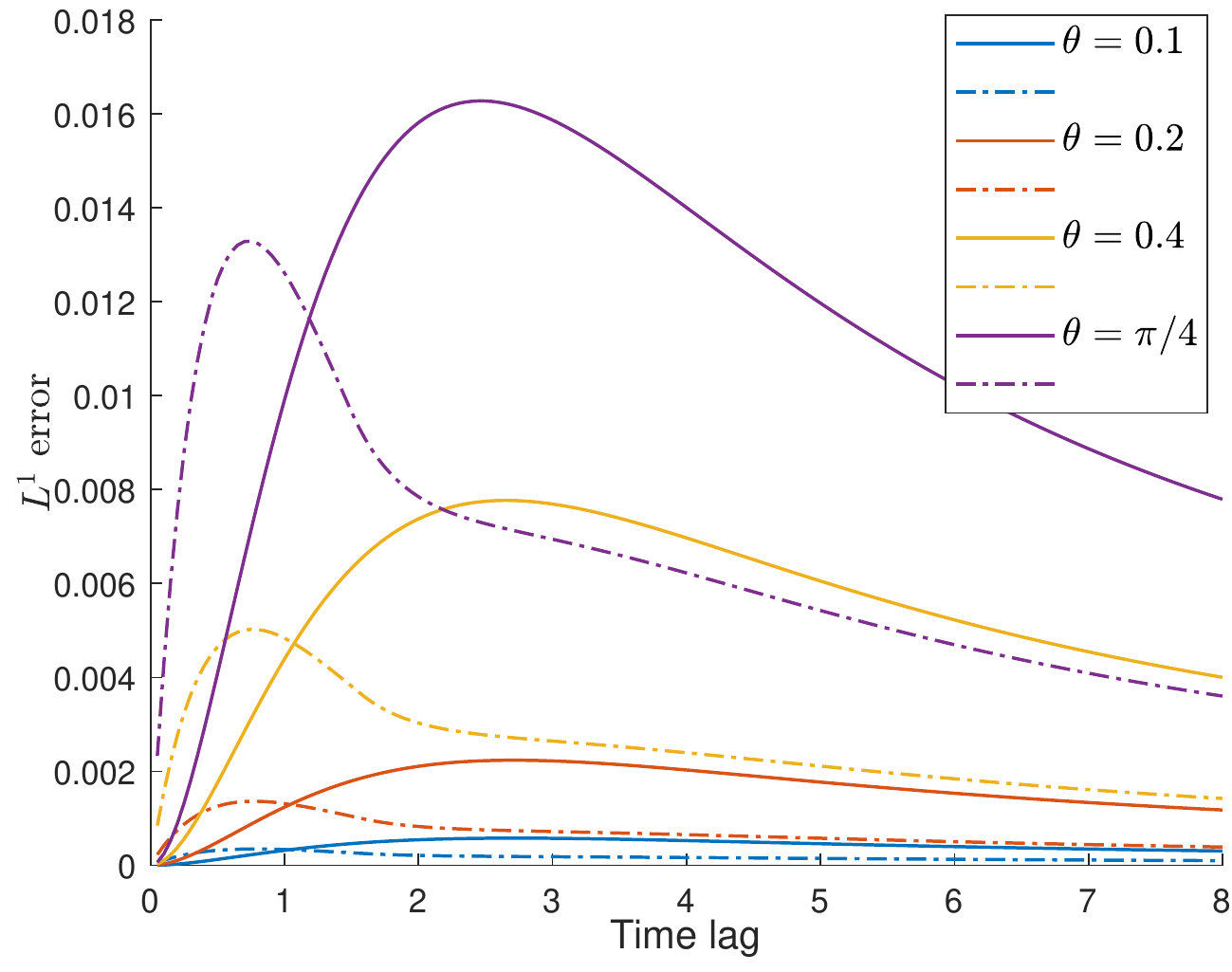}}
           \caption{The absolute error and the $L^{1}$ time-averaged error in the equilibrium autocovariance functions for Approach~1 ({\bf solid line}) and for Approach~2 ({\bf dash-dot line}), plotted over different time lags $\tau=|t-s|$. In this case, $\lambda = 2$ and $\theta$ is varying.
        }\label{fig:ACF_LagTime}
    \end{figure}  

    To explore this observation further,
    the following theorem provides error estimates in the case of a fixed $\tau$ as $\lambda\to+\infty$, which is the regime in which the timescale separation in the system grows very large.

    \begin{theorem}\label{th:2Derror-longtime}
        In the two-dimensional setting described above, as $\lambda\to+\infty$, the error between the true equilibrium covariance and the equilibrium covariance for Approach 1 satisfies
        \begin{equation}\label{App1_GlobalErr}
        \begin{split}
        |R(\tau)-R_1(\tau)|= \beta^{-1}(\e^{-\tau}-\e^{-\tau\sec^2\theta})\cos^2\theta+O(\beta^{-1}\lambda^{-1})
        \\[2mm]
        \text{and}
        \quad
        \frac{|R(\tau)-R_1(\tau)|}{|R(\tau)|}=1-\e^{-\tau\tan^2\theta}+O(\lambda^{-1})
        \end{split}
        \end{equation}
        for all $\tau\gg \lambda^{-1}$.
        In particular, at $\tau=1$, we have        \begin{equation}\label{App1_RelErr_Lag1}
        \frac{|R(1)-R_1(1)|}{|R(1)|}=1-\e^{-\tan^2\theta}+O(\lambda^{-1}).
        \end{equation}
        In the same regime, the error between the true equilibrium covariance and the equilibrium covariance for Approach~2 satisfies
        \begin{equation}\label{App2_GlobalErr}
        \begin{gathered}
|R(\tau)-R_2(\tau)|= \beta^{-1}\frac{\sin^2\theta}{\lambda}\big|\e^{-\lambda \tau}-\e^{-\tau }(1-\tau)\big|+O(\beta^{-1}\lambda^{-2}),
        \\[2mm]
        \text{and}
        \quad
        \frac{|R(\tau)-R_2(\tau)|}{|R(\tau)|}=\min\left\{1,\, \frac{|\tau-1|}{\lambda}\tan^2\theta + O(\lambda^{-2})\right\}.
        \end{gathered}
        \end{equation}
        Moreover, at $\tau=1$, we have        \begin{equation}\label{App2_RelErr_Lag1}
         \frac{|R(1)-R_2(1)|}{|R(1)|}= \frac{\tan^2\theta|1-\frac12\tan^2\theta|}{\lambda^2}+O(\lambda^{-3}).
        \end{equation}
    \end{theorem}

    \noindent
    A proof of this result is given in Appendix~\ref{app:proof-2Derror-longtime}. The main idea is to compute the autocovariance functions explicitly in this case, and perform asymptotic expansions of the resulting expressions.

    We note a series of features of these results. First, all of the error expressions are exactly zero when $\theta=0$. This reflects the fact that in this case, we have selected CG variable which is perfectly aligned with the eigenspace corresponding to the slowest time-scale, diagonalising the system, as was discussed in Section~\ref{sec:equilibrium}. Second, when $\theta\to\pm\frac{\pi}{2}$, the relative error grows large for both approaches. This corresponds to the selecting a CG variable which is very poorly aligned with the direction which exhibits the slowest time-scale of evolution in the system.
    Third, we note that for fixed $\tau$ the relative error for Approach~1 saturates, while it decays rapidly irrespective of $\theta$ in Approach~2. This demonstrates the fact that while Approach~1 provides more accurate covariance behaviour for very short time-lags $\tau$, at larger time-lags, 
    Approach 2 behaves better. This provides a confirmation of the pattern observed in Figure~\ref{fig:ACF_LagTime}, where Approach~2 provides greater accuracy than Approach~1 at large time lags.

Considering the asymptotic expansions of relative autocovariance errors at $\tau=1$ in terms of spectral gap $(\lambda-1)$, \eqref{App1_RelErr_Lag1} suggests that the relative error when using Approach~1 saturates as $(\lambda-1)\to \infty$, while \eqref{App2_RelErr_Lag1} suggests that the relative error when using Approach~2 decays with rate $(\lambda-1)^{-2}$.  In practice, the numerical experiments we perform in Figure~\ref{fig:MaxRelErrRate_vs_lambda} agree with the analytical predictions for various $\theta$. 
 This demonstrates that Approach 1, while accurately capturing static equilibrium statistics, poorly captures dynamical equilibrium statistics in cases where there is large scale separation. This is notable since the case of large scale-separation is usually viewed as the regime in which this CG approach works best; see for example \cite{Legoll2012}.

 The results of Theorem~\ref{th:2Derror-longtime} are only valid at lag times which are long relative to $\lambda^{-1}$. For shorter lag times, we have the following result:
 
    \begin{theorem}\label{th:2Derror-shorttime}
        For $\tau\ll \lambda^{-1}\ll 1$, the absolute and relative autocovariance error for Approach~1 behave asymptotically as follows:
        \begin{gather}
        |R_1(\tau)-R(\tau)| = \tfrac12\beta^{-1}\tau^2(\lambda-1)\sin^2\theta+O(\tau^2)+O(\lambda^3\tau^3)
        \label{ACF_App1_short}\\[2mm]
        \text{and}\quad
        \frac{|R_1(\tau)-R(\tau)|}{R(\tau)} = \tfrac12\tau^2(\lambda-1)\tan^2\theta+O(\tau^2)+O(\lambda^3\tau^3).\label{ACF_App2_short}
       \end{gather}
            In the same regime, the absolute and relative autocovariance error for Approach~2 behave asymptotically as follows:
        \begin{gather}
        |R_2(\tau)-R(\tau)|= \beta^{-1}\tau\sin^2\theta+O(\tau\lambda^{-1},\lambda\tau^2)\\[2mm]
        \text{and}\quad
        \frac{|R_2(\tau)-R(\tau)|}{R(\tau)} = \tau \tan^2\theta+ O(\tau\lambda^{-1},\lambda\tau^2).
        \end{gather}
    \end{theorem}

We note that the range of time-lags where these asymptotic results are valid shrinks as $\lambda\to +\infty$. Within this regime however, Approach~1 performs better than Approach~2, but it is questionable whether this insight is of general use, since the ACF behaviour for larger time-lags dictates the system response. Theorem~\ref{th:2Derror-longtime} would therefore seem to be the result of greater relevance for assessing the accuracy of the Approaches in practice.

Finally, we note that although we do not provide analytical expressions for the $L^1(0,\tau)$ mean absolute or relative errors in Theorem~\ref{th:2Derror-longtime}, similar results can be obtained by carefully matching the asymptotic results above, and then integrating the absolute error expressions in time. Since the resulting expressions are rather complex to interpret, we focus instead on a numerical exploration in Section~\ref{sec:num}.

     \subsection{Summary}\label{subsec:summary}
     We now summarise the results of the analysis of two-dimensional systems performed in this section. If we assume that the  spectral gap is large, i.e. $\lambda-1\gg 1$, then the relative error behaves as follows:
     \begin{itemize}
      \item For short lag times $\tau\ll \lambda^{-1}\ll 1$, we have     \begin{equation}\label{ACF_short_summary}
     \begin{split}
     &\text{\bf Approach 1:} \quad 
     \frac{|R_1(\tau)-R(\tau)|}{|R(\tau)|}=
     \tfrac12\tau^2(\lambda-1)\tan^2\theta
     +O(\tau^2,\lambda^3\tau^3),\\
      &\text{\bf Approach 2:} \quad
              \frac{|R_2(\tau)-R(\tau)|}{|R(\tau)|}= \tau\tan^2\theta+O(\tau\lambda^{-1},\lambda \tau^2).
     \end{split}
     \end{equation}
\item For fixed lag time $\tau =1$, we have
\begin{equation}\label{ACF_long_summary}
     \begin{split}
     &\text{\bf Approach 1:} \quad \frac{|R(1)-R_1(1)|}{|R(1)|}=1-\e^{-\tan^2\theta}+O(\lambda^{-1}),\\
  &\text{\bf Approach 2:} \quad 
\frac{|R(1)-R_2(1)|}{|R(1)|}=
\frac{\tan^2\theta|1-\frac12\tan^2\theta|}{\lambda^2}
     \end{split}
     \end{equation}
\end{itemize}     

     These predictions suggest that Approach~2 performs better than Approach~1 when both the spectral gap and deviation from the optimal coarse-graining projection are considered in the simulation.

\section{Numerical results}\label{sec:num}

To complete our investigation of the approximation approaches we have proposed, we consider three numerical examples to validate our analytical results. In turn, we study:
\begin{enumerate}
    \item A two-dimensional system of the type studied analytically in Section~\ref{sec:2d_error}. We compare our analytical predictions with numerical results in Subsection~\ref{subsec:2D_num}.
    \item A ten-dimensional system in which we study the result of progressively coarse-graining a system, which was the subject of Corollary~\ref{cor:coarsening}. These results are given in Subsection~\ref{subsec:multiD_num}.
    \item A forty-dimensional harmonic chain model with three choices of spring stiffnesses, in which we compare the absolute and relative autocovariance errors, \eqref{Abs_error} and \eqref{Rel_Abs_error} for both Approach~1 and 2 for various choices of model parameters. This example is studied in Subsection~\ref{subsec:1Dchain_num}.
\end{enumerate}
In all cases, we find that the new coarse-graining Approach~2 we have proposed more accurately reflects dynamical statistics than Approach~1.
    
    \subsection{Two-dimensional study}\label{subsec:2D_num}
    
    We firstly recall the simple 2D system studied in Section~\ref{sec:2d_error}, given by
    \begin{equation}\label{2D_driftA}
    \bA=\begin{pmatrix}
    1 & 0\\ 0 &     \lambda
    \end{pmatrix}\quad\text{with}\quad\bPhi = \begin{pmatrix}
         \cos\theta & \sin\theta 
    \end{pmatrix}.
    \end{equation}
    We set an initial condition for the full dynamics $(x_1,\, x_2)=(5,\,-4)$, and note that the coarse-grained variable is $\xi=x_1\cos\theta+x_2\sin\theta$.  For the approximate dynamics, the initial conditions are set to be deterministic, with $\xi_0=5\cos\theta-4\sin\theta$ and $\zeta_0=-A_1^{-1}\alpha^*\xi_0$, and we set the inverse temperature $\beta =1$. Simulations were performed using the Euler-Maruyama scheme with total simulation time of $T=60$ and time-step size $\Delta t=5\times10^{-4}$. When computing the sample ACF, we discard the first half of each trajectory when computing the ACF to ensure our numerical results are close to the equilibrium ACFs.
    
    \paragraph{Validation of autocovariance formulae.}
We first set $\lambda=20$, $\theta=0.3$ in \eqref{2D_driftA} and numerically validate the analytical autocovariance functions (ACFs) computed for Approach~1 and Approach~2, as summarized in Table~\ref{tab:summary}. 
    For Approach~1 and Approach~2, the ACF was computed for $5000$ independent sample trajectories $\xi_t$ with time lags $\tau\in(0,\,2]$, and these results were averaged to compare the empirical values with the associating analytical formula. The results are summarized in Figure~\ref{fig:2D_TrajACF}. We see that the sample averages are consistent with our analytical expressions. We also see that over the short time-scale considered, the ACF of Approach~1 consistently underestimates the true ACF whereas the ACF Approach~2 remains closer to that of the full dynamics.
    We also see from this figure that Approach~1 produces smaller errors when $\tau\ll 1$ but deviates more when $\tau$ becomes large, whereas Approach~2 provides a better approximation to the ACF over longer time-scales.
    
    \begin{figure}[htp!]
        \centering
           {   \includegraphics[width=0.5\textwidth]{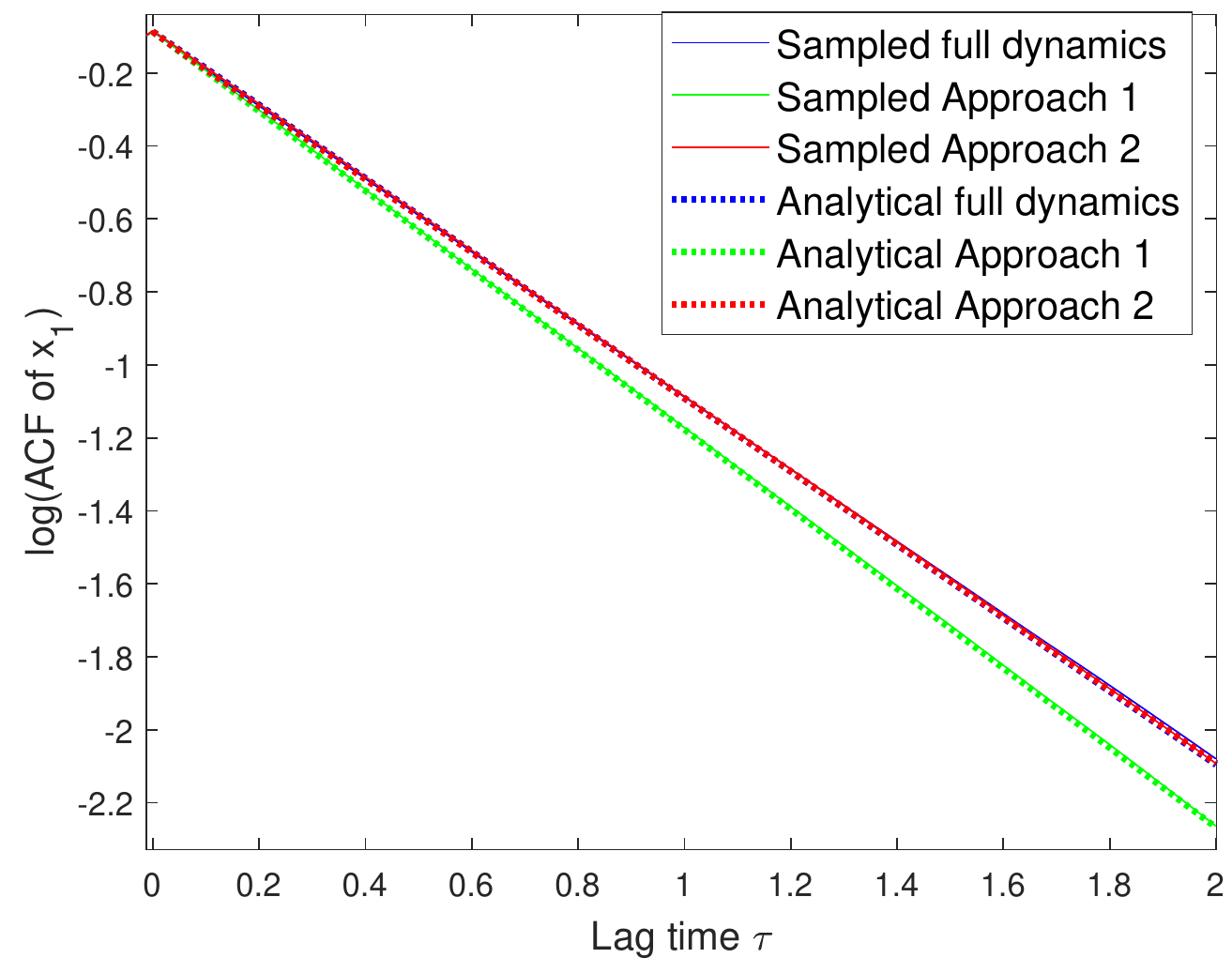}}
        \caption{Comparison of the log of the ACF of the coarse-grained variable derived from the full dynamics, Approach~$1$, and Approach~$2$. In each case, analytical formulae are compared with the results of Monte Carlo sample autocovariances, and are seen to provide a match to within sampling error. Model parameters are $\lambda=20$, $\theta=0.3$ with $\bA$,  $\bPhi$ as defined in \eqref{2D_driftA}. Monte Carlo averages are taken over 5000 samples simulated using the Euler-Maruyama scheme with total simulation time $T=60$ and time-step $\Delta t=5\times10^{-4}$.
        }
        \label{fig:2D_TrajACF}
    \end{figure}
    
\paragraph{Asymptotic error performance with varying spectral gap.}
Next, to study the sharpness the errors predicted by Theorem~\ref{th:2Derror-longtime} and Theorem~\ref{th:2Derror-shorttime}, 
we select $\theta \in\{0.05,\,0.2,\, 0.4,\,\frac{\pi}{4}\}$ in \eqref{2D_driftA} and vary $\lambda\in [1.1,\,1000]$. 

We first study the dependence of relative errors on the spectral gap $(\lambda-1)$ for Approach~1 and Approach~2. Both the regime where $\tau\to 0$ and where $\tau = 1$ were considered, and these results are shown in Figure~\ref{fig:MaxRelErrRate_vs_lambda}.
Notice that the numerical values and asymptotic rates as the spectral gap grows large are very close to the theoretical estimates in \eqref{ACF_short_summary} and \eqref{ACF_long_summary}. 
\begin{figure}[htp!]
    \centering
   {\includegraphics[width = 0.9\textwidth]{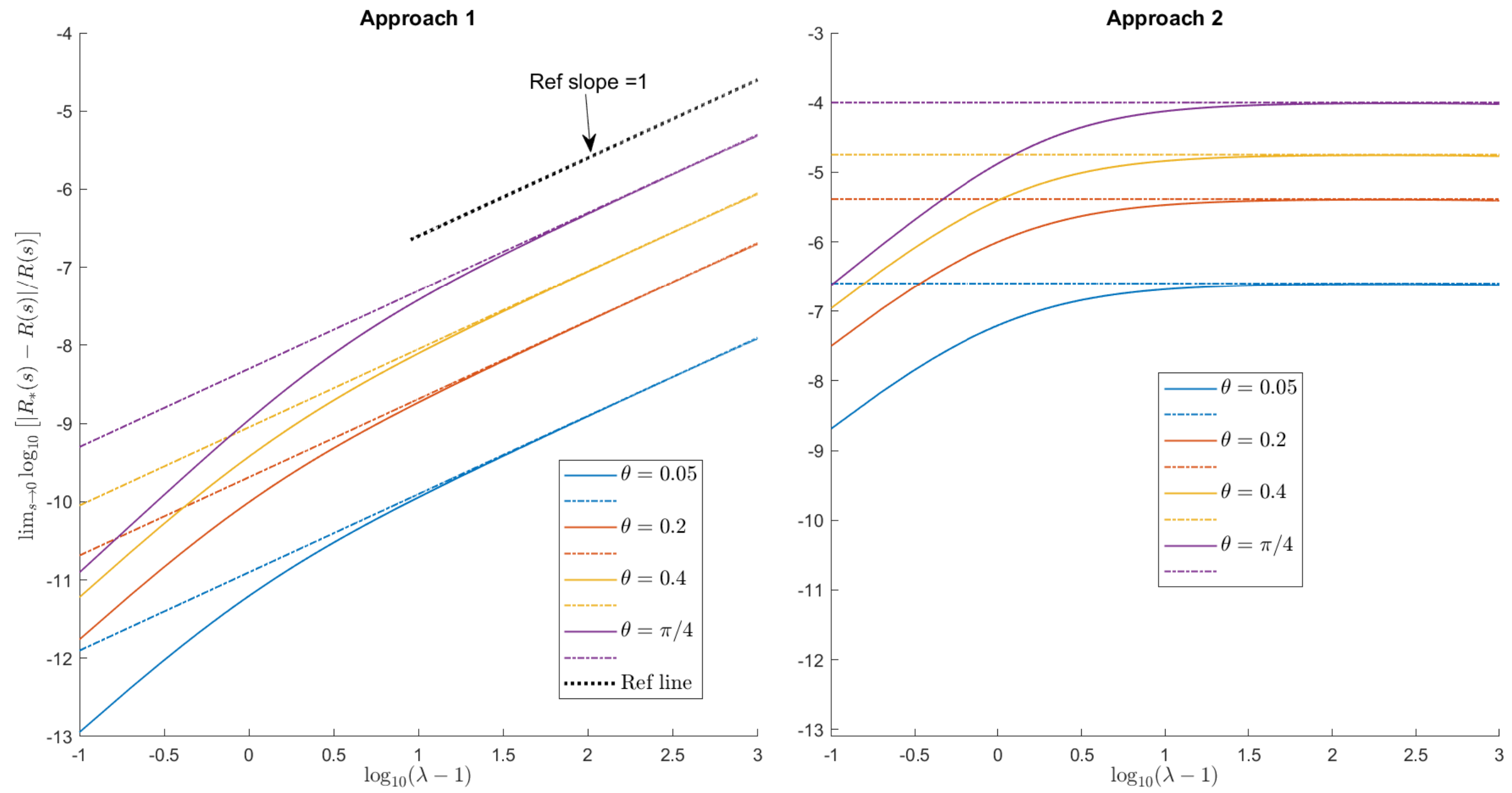}}
  { \includegraphics[width = 0.86\textwidth]{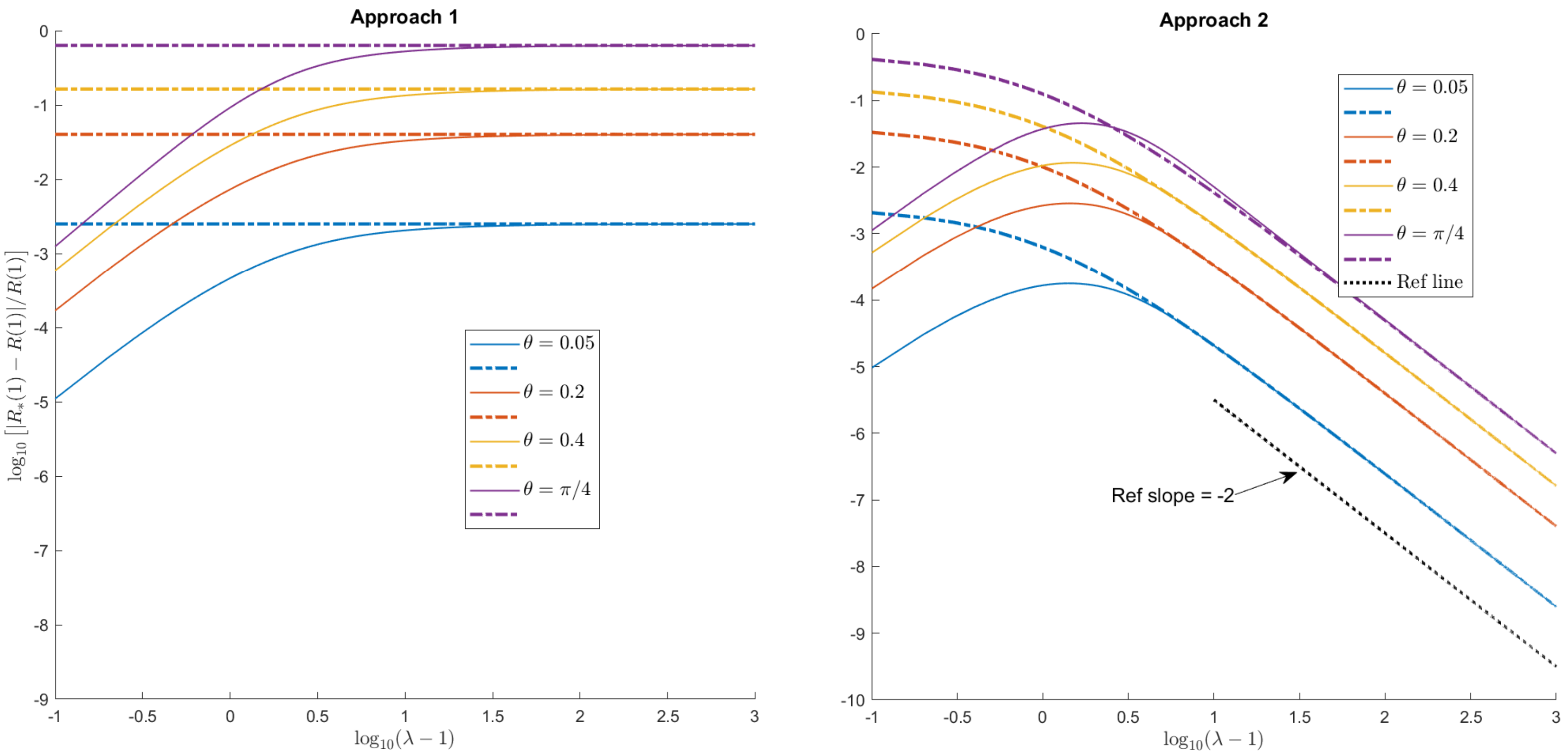}}
    \caption{Relative autocovariance errors for Approach~1 and Approach~2 as $\lambda$ varies in a two-dimensional system. {\bf Top row:} error for short lag times $\tau\to 0^{+}$. {\bf Bottom row:} error for fixed lag time $\tau=1$. In all four plots, solid lines denote the true numerical values and dash-dotted lines denote the theoretical estimates of Approach~1 and Approach~2 in \eqref{ACF_short_summary} and \eqref{ACF_long_summary}. Different colours reflect different choices of $\theta$ made in \eqref{2D_driftA}.
    }
\label{fig:MaxRelErrRate_vs_lambda}
\end{figure}

Secondly, in Figure~\ref{fig:AbsErr_vs_tau} we plot the absolute errors against different time-lags $\tau$ for both approaches and for fixed $\lambda = 10$. We note that the analytical estimates for both short and global time lags, as presented in Theorem~\ref{th:2Derror-longtime} and Theorem~\ref{th:2Derror-shorttime}, provide accurate error bounds for both approaches. In addition, the short time estimate  \eqref{ACF_App1_short} is very close to Approach~1 while the global time estimate \eqref{App2_GlobalErr} is close to Approach 2.
\begin{figure}[htp!]
    \centering
  { \hspace{-7mm}\includegraphics[width = 0.87\textwidth]{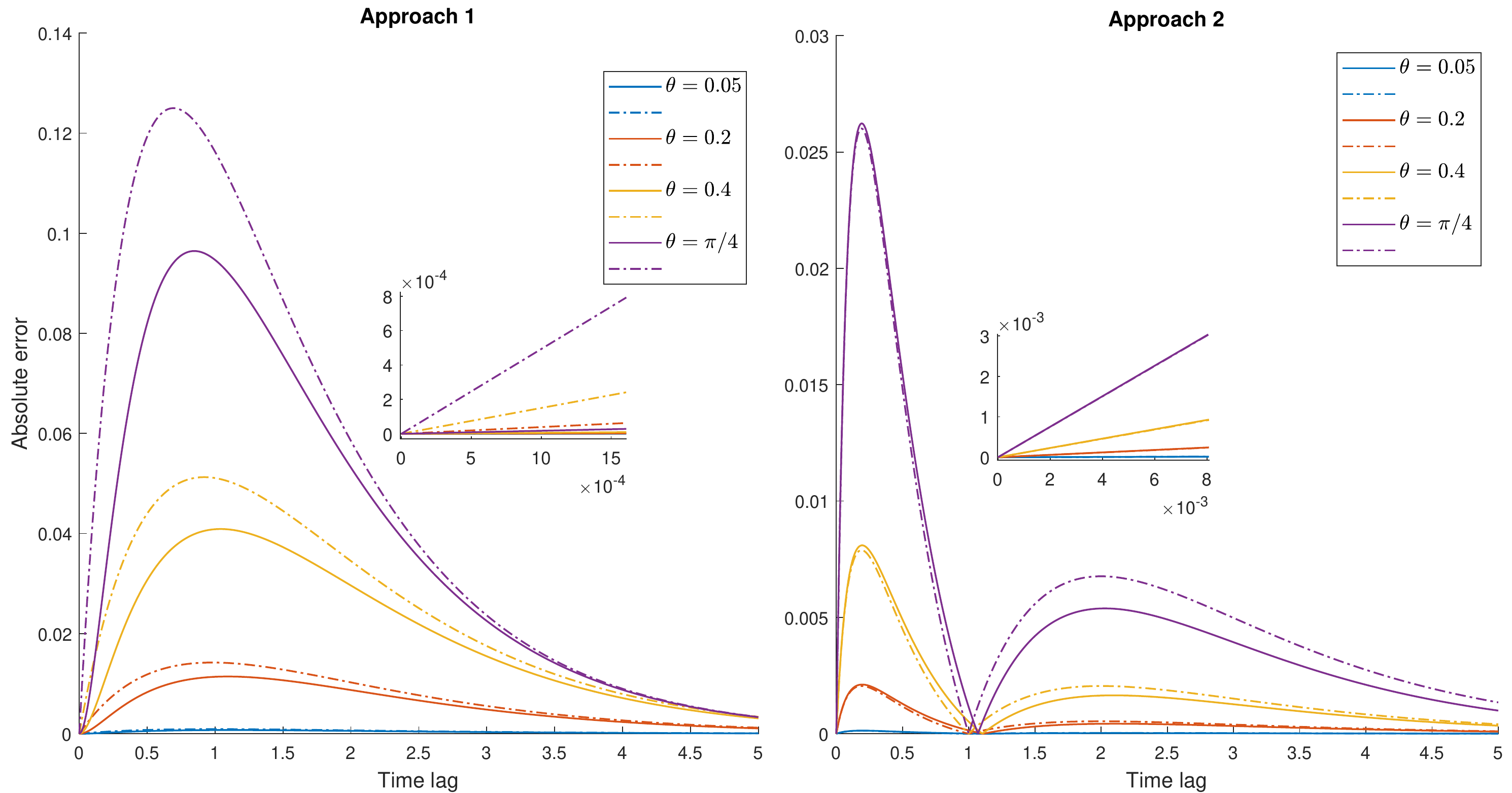}}
   { \includegraphics[width = 0.9\textwidth]{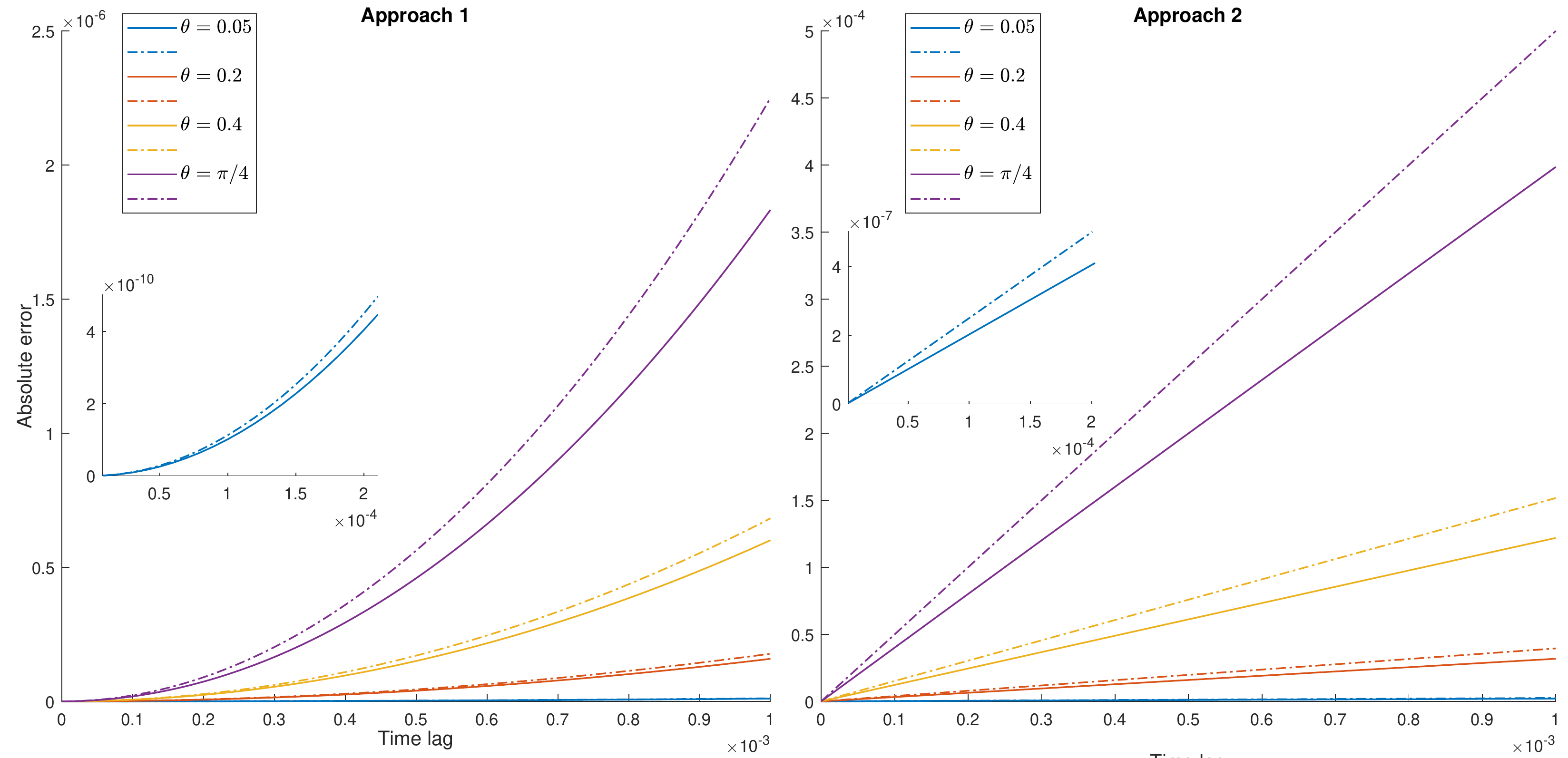}}
  
    \caption{The absolute autocovariance errors of  Approach~1 and Approach~2 plotted against varying time-lag $\tau$ for the two-dimensional system described in \eqref{2D_driftA} with $\lambda = 10$.  {\bf Top row:} comparison between the error over longer lag times $\tau\in [0,\,5]$ and estimates in Theorem~\ref{th:2Derror-longtime}.  {\bf Bottom row:} comparison between the error for short lag time $\tau\in[0,\, 0.001]$ and estimates in Theorem~\ref{th:2Derror-shorttime}.
    Solid lines denote true error and dash-dot lines denote the theoretical estimates of Approach~1 and Approach~2 the ACF error given in \eqref{App1_GlobalErr} and \eqref{App2_GlobalErr} for the {top row}; and in \eqref{ACF_App1_short} and \eqref{ACF_App2_short} for the bottom row. Different colours reflect different choices of $\theta$ in \eqref{2D_driftA}.
    }
    \label{fig:AbsErr_vs_tau}
\end{figure}
Both analytical results and simulations confirm that Approach~2 performs better than Approach~1 in the large time-scale separation regime.

\subsection {Multi-dimensional study}
\label{subsec:multiD_num}
Next, we consider a multi-dimensional system with $\bA$ to be a tridiagonal matrix with increasing diagonal entries $\lambda_1\le \lambda_2\le \dots\le \lambda_{10}$:
\begin{equation}\label{eq:10D_system}
\bA=\begin{pmatrix}
\lambda_1 & \sigma & 0 & \dots & 0\\
\sigma & \lambda_2 & \sigma &\ddots &\vdots\\
0 & \sigma & \lambda_3 & \ddots & 0 \\
\vdots & \ddots &\ddots &\ddots &\sigma \\
0 & \dots & 0 & \sigma & \lambda_N 
\end{pmatrix}.
\end{equation}
To parametrise a scale separation in the system that we can manipulate, we choose $\lambda_{1}=1$ and $N=10$, and $\lambda_2,\dots, \lambda_{10}$ are set to be equally spaced on $[1.5,\, 10]$. The parameter $\sigma>0$ allows us to perturb the eigenvalues and corresponding eigenvectors for the matrix away from $(\lambda_i,\be_i)$ for the case $\sigma=0$. We treat this perturbation as being a form of high-dimensional analogue of the rotation considered in Section~\ref{sec:2d_error} and Subsection~\ref{subsec:2D_num}, with the caveat that in general, the eigenvalues will also be perturbed in this case. In this setting, we set the dimension of the coarsest coarse-graining to be $d=1$ with the coarse-grained variable to be $\xi = q_1$, so that
$
\Phi= \begin{pmatrix}
1 & 0 &\cdots &0
\end{pmatrix}
$. This choice entails that as $\sigma$ increases, the coarse-graining projection deviates more from exactly selecting the eigenspace corresponding to the lowest eigenvalue, resulting in an increased difference between observations made from the full system and those taken from approximation Approach~1 and Approach~2 (solid lines in Figure~\ref{fig:Dim10_ACF_Rel}). We plot the ACF relative errors of $q_1$ of the full system with that of the finest coarse-grained system with $d=1$. Consistent with our observations from two-dimensional systems, the maximum relative error for Approach~2 is significantly smaller than that of Approach~1 in all cases we tests. The relative error for Approach~1 is also observed to grow faster than that for Approach~2 when the projection deviates from the optimal eigenspace.

\begin{figure}[htp!]   
\centering
        { \includegraphics[width = 0.45\textwidth]{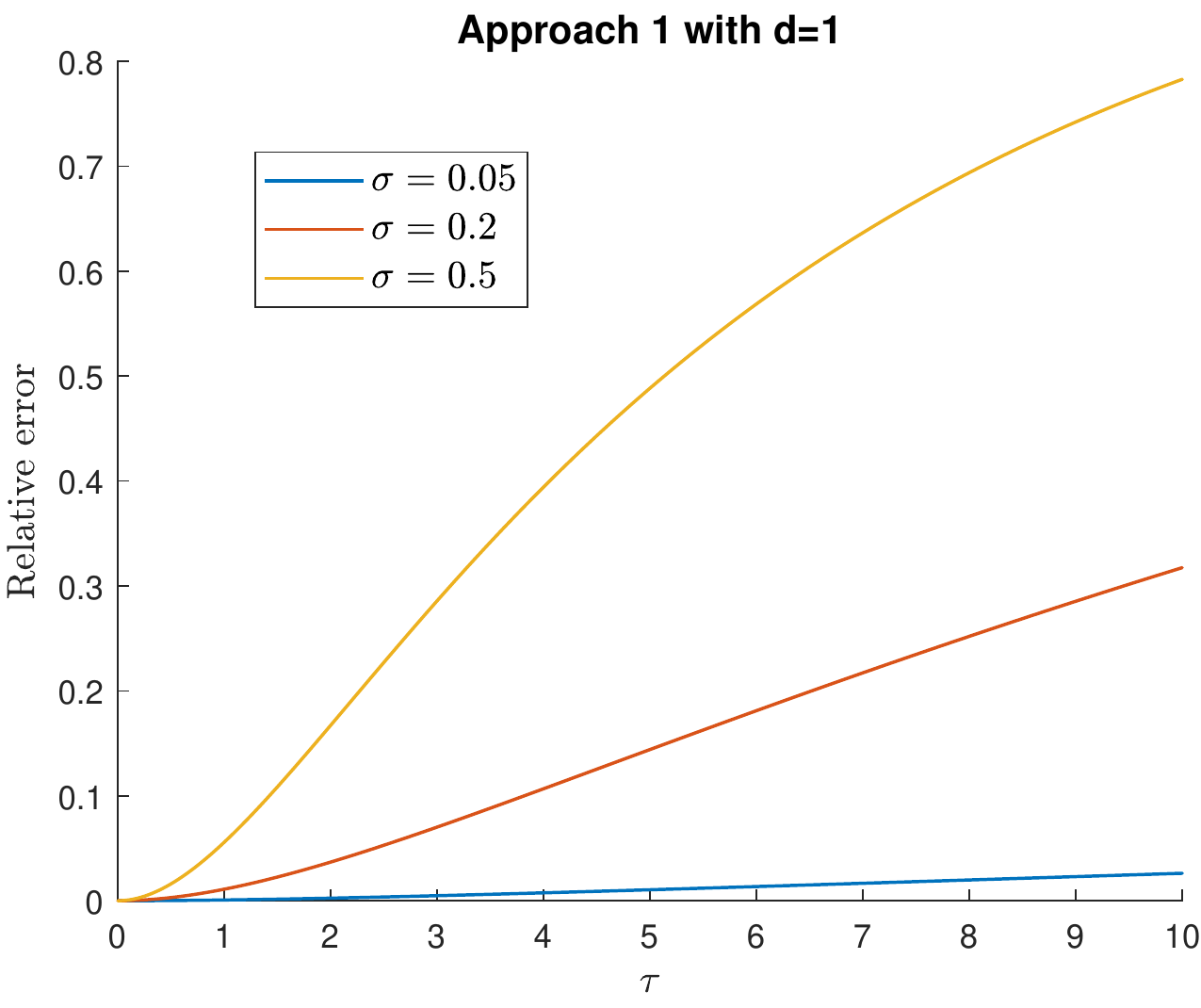}}
      \caption{Relative ACF errors of coarse-graining the first coordinate of the ten-dimensional system driven by $\bA$ given in \eqref{eq:10D_system}. 
   {\bf Left figure:} results of Approach~1. {\bf Right figure:} results of Approach~2.
   In the ten-dimensional setting, we have $\lambda_1 = 1$, $\lambda_2 = 1.5$, $\lambda_{10}=10$ and all other $\lambda_i$ are equally distributed in between $1.5$ and $10$. The increasing $\sigma$ results in larger deviations from the eigenspace.}
   \label{fig:Dim10_ACF_Rel}
 \end{figure}

Meanwhile, we investigate the effects of progressively coarse-graining by employing a wide range of coarse-graining levels with intermediate coarse dimension $n=2,\,4,\,6,\,8$. We fix $\sigma = 0.5$ and compare the ACF of $q_1$ among different levels of coarse-graining. The second column of Figure~\ref{fig:Dim10_ACF} displays the semi-log plots of the relative ACF errors to the full system and the third column of Figure~\ref{fig:Dim10_ACF} displays the difference between simulation of $d=1$ reduced system and simulations of intermediate $n$ system. The results of Approach~1 are consistent with Corollary~\ref{cor:coarsening} as the finest coarsening with $d=1$ provides the upper bound of all errors. For Approach~2, we confirm that the strict monotonicity in the error shown for Approach~1 does not hold, although we do observe that the difference of relative errors between differing levels of coarse-graining is very small.

\begin{figure}[htp!]
    \centering
        { \includegraphics[width = 0.45\textwidth]{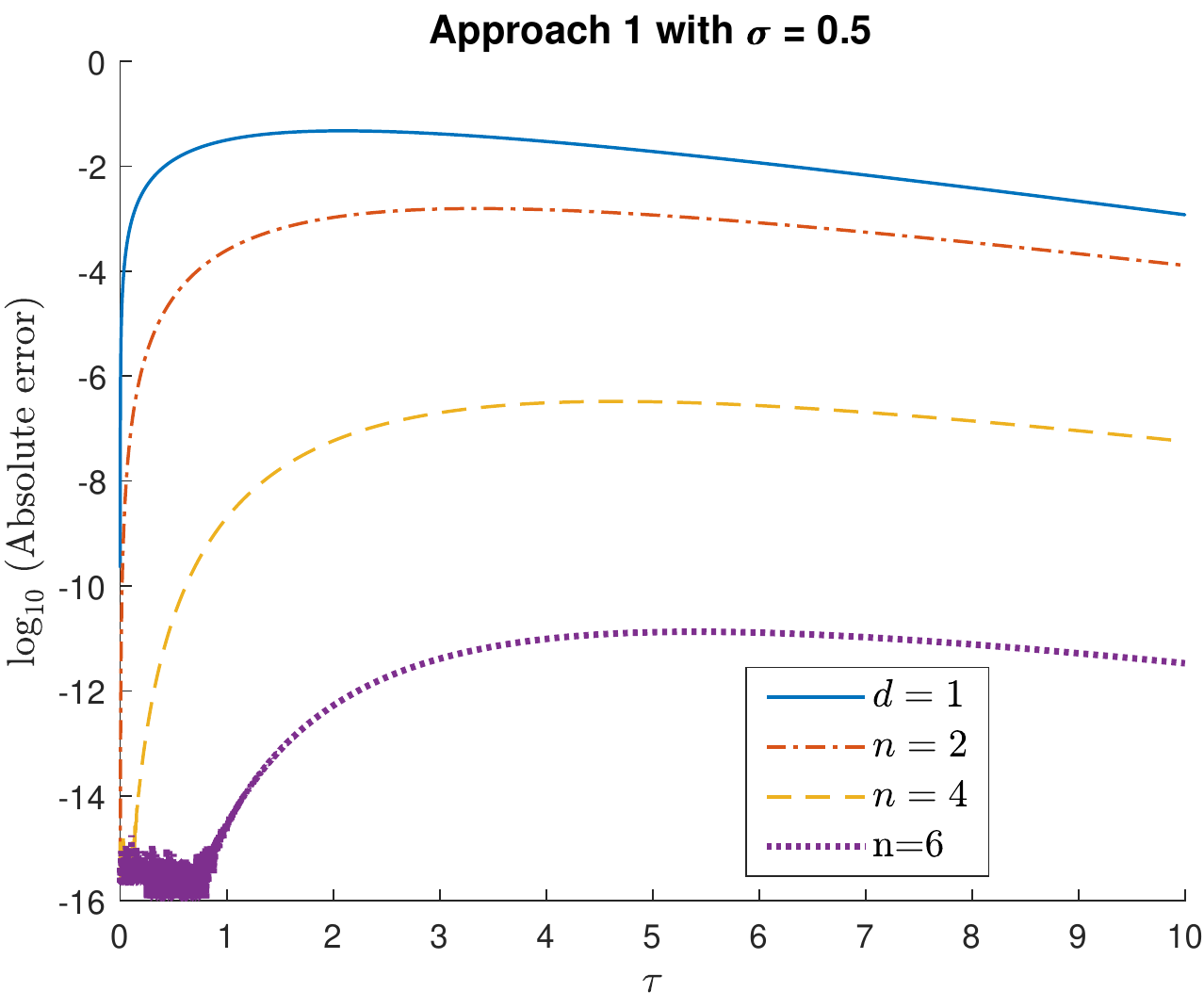}}
     { \includegraphics[width = 0.45\textwidth]{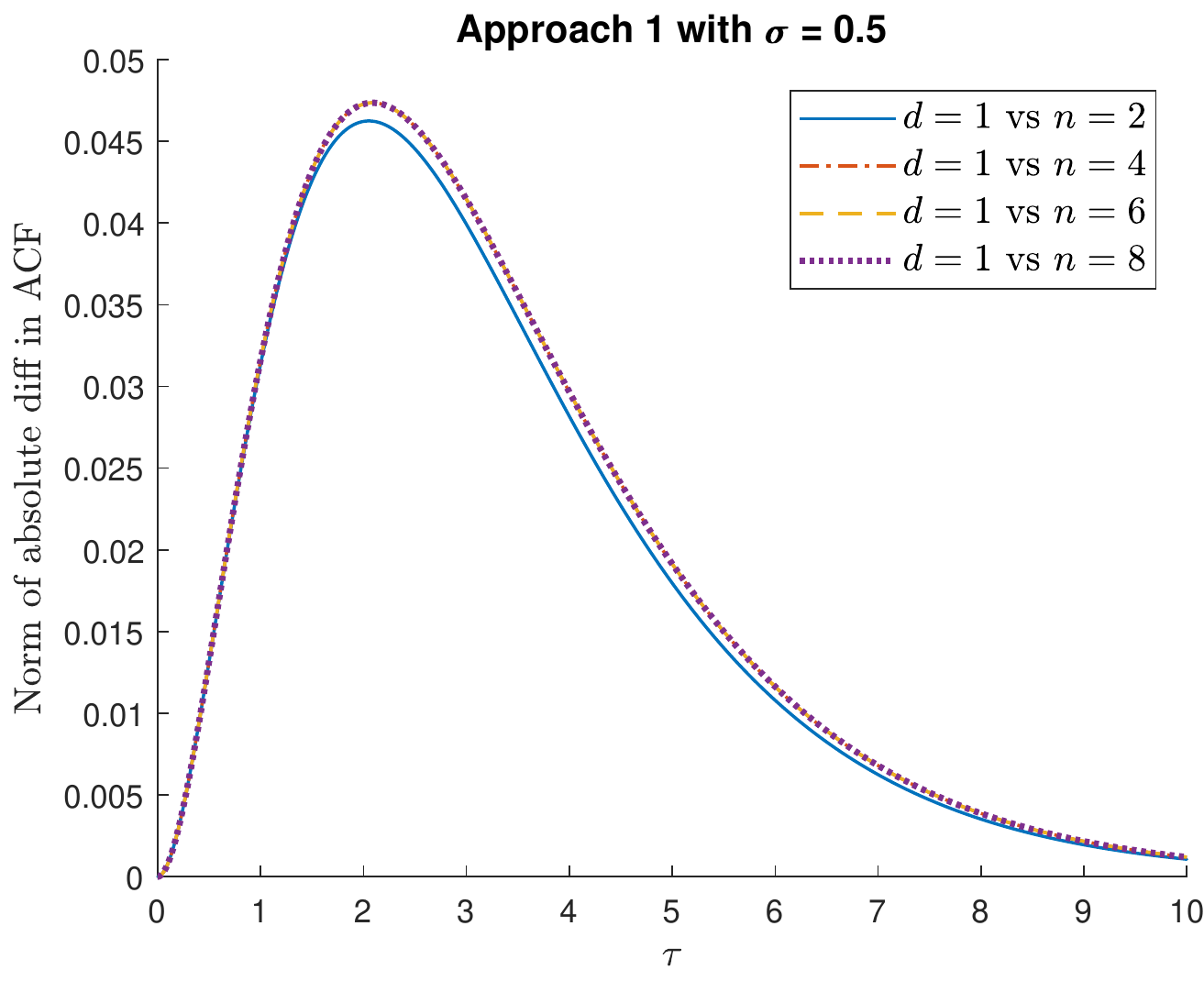}}
     { \includegraphics[width = 0.45\textwidth]{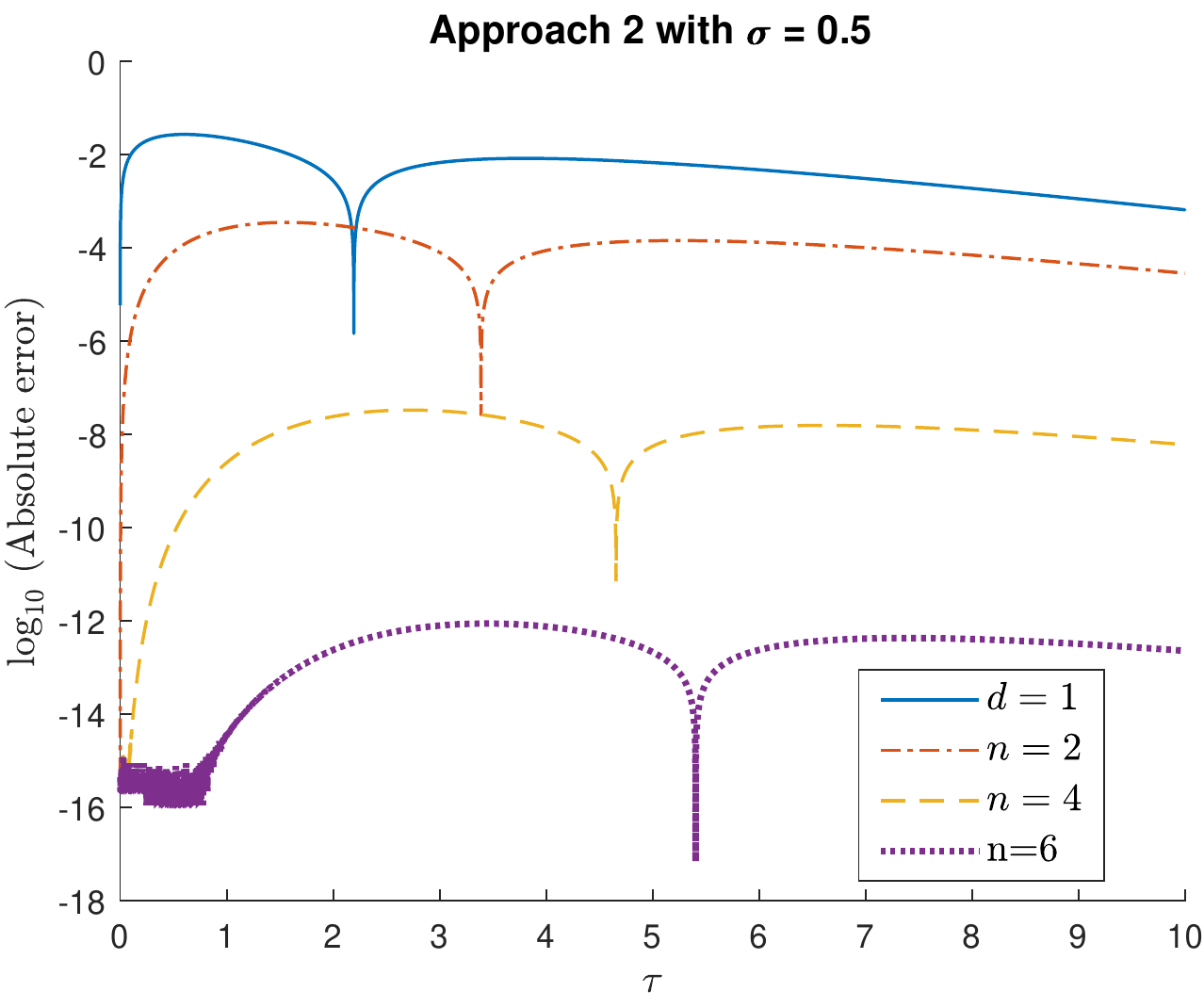}}
  { \includegraphics[width = 0.45\textwidth]{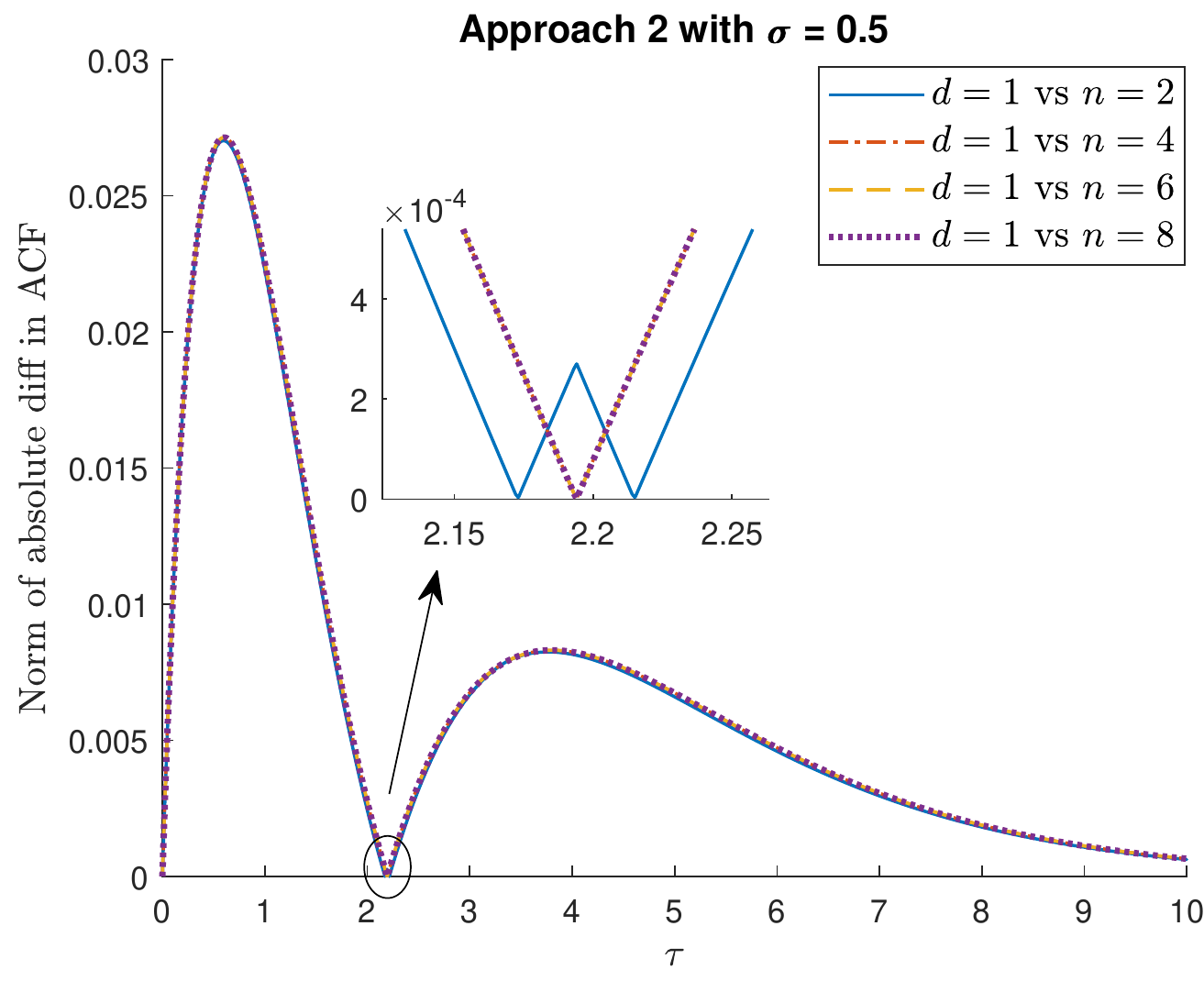}}
       \caption{ACF absolute errors of coarse-graining  the first coordinate 
       of the ten-dimensional system driven by $\bA$ given in \eqref{eq:10D_system}. 
    {\bf Top row:} results of Approach~1. {\bf Bottom row:} results of Approach~2.
    In the ten-dimensional setting, we have $\lambda_1 = 1$, $\lambda_2 = 1.5$, $\lambda_{10}=10$ and all other $\lambda_i$ are equally distributed in between $1.5$ and $10$. First column figures correspond to the $\log_{10}$-scale errors \eqref{Upper_bound_Coarest} in different partially coarse-grained simulations, where the results of $n=8$ is dropped as it is already machine prison. Second column figures represent the norm of absolute difference in ACF \eqref{Coarest_vs_intermediate} between the $d=1$ and intermediate (various $n$) coarse-grained dynamics.
    }
    \label{fig:Dim10_ACF}
 \end{figure}
    
\subsection{One-dimensional harmonic spring-mass system}\label{subsec:1Dchain_num} 
In our final set of numerical results, we study the performance of our different approaches as well as different level of coarsening for a more practical problem.
    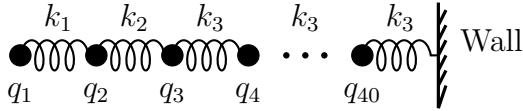
\begin{figure}[htp!]
    \centering
 \begin{tikzpicture}
 \draw (-4,0) circle(4pt) [fill];
 \node at (-4, -0.5) {$q_1$};
  \node at (-3.5, 0.5) {$k_1$};
  \draw[thick,decoration={aspect=0.5, segment length=2mm, amplitude=2mm,coil},decorate] (-4,0)--(-3,0);
 \draw (-3,0) circle(4pt) [fill];
  \node at (-3, -0.5) {$q_2$};
 \draw[thick,decoration={aspect=0.5, segment length=2mm, amplitude=2mm,coil},decorate] (-3,0)--(-2,0);
   \node at (-2.5, 0.5) {$k_2$};
  \draw (-2,0) circle(4pt) [fill];
    \node at (-2, -0.5) {$q_3$};
  \draw[thick,decoration={aspect=0.5, segment length=2mm, amplitude=2mm,coil},decorate] (-2,0)--(-1,0);
   \node at (-1.5, 0.5) {$k_3$};
  \draw (-1,0) circle(4pt) [fill];
    \node at (-1, -0.5) {$q_4$};
  \draw (-0.5,0) circle(1pt) [fill];
    \draw (-0.25,0) circle(1pt) [fill];
     \node at (-0.25, 0.5) {$k_3$};
      \draw (0,0) circle(1pt) [fill];
        \draw (0.5,0) circle(4pt) [fill];
          \node at (0.5, -0.5) {$q_{40}$};
      \draw[thick,decoration={aspect=0.5, segment length=2mm, amplitude=2mm,coil},decorate] (0.5,0)--(1.5,0);
        \node at (1, 0.5) {$k_3$};
      \draw[very thick,black] (1.5,-0.7)--(1.5,0.7);    
 \draw[very thick,black] (1.5,-0.7)--(1.6,-0.4);
 \draw[very thick,black] (1.5,-0.4)--(1.6,-0.2);
 \draw[very thick,black] (1.5,-0.2)--(1.6,0);
\draw[very thick,black] (1.5,0)--(1.6,0.2);
 \draw[very thick,black] (1.5,0.2)--(1.6,0.4);
 \draw[very thick,black] (1.5,0.4)--(1.6,0.6);
   \node at (2.2, 0.2) {Wall};
 \end{tikzpicture}
 \caption{Illustration of the 1D harmonic spring-mass system considered; $q_i$ is the displacement of each mass from its equilibrium position. The left boundary is left free, and the right boundary is assumed to be fixed. The relevant spring constant for each spring in the chain is shown.\label{Fig:spring-mass} } 
\end{figure}    

We consider a one-dimensional spring-mass system connected with different types of springs, as demonstrated in Figure~\ref{Fig:spring-mass}. 
 We set the total number of masses to be $N=40$ as a reference model, and consider the intermediate coarse-graining to be $n=4$ and the ultimate coarse-graining to be $d=2$, so that the ultimate coarse-grained
variables of interest are the first two masses $q_1$ and $q_2$.  Recalling Subsection~\ref{sec:multiple_reduction}, we explore the impacts of different levels of coarsening. We therefore consider the following coarse-graining scenarios:
\begin{enumerate}
    \item The first option is a direct reduction from $N=40$ to $d=2$. We then compare the statistics of $(q_1,\,q_2)$ as produced by the $40$D full dynamics with those by the surrogate $2$D dynamics by either Approach~1 or Approach~2. We use `40D vs 2D' to represent the results of this option.
     \item The second option is a reduction from $N=40$ to $n=4$. We firstly approximate the full dynamics by either Approach~1 or Approach~2 to reduce the number of masses from $N=40$ to $n=4$. Once we get the data for $(q_1,\,\dots, q_4)$ from this approximations, we then compare the statistics of $(q_1,\,q_2)$ by this $4$D surrogate dynamics with the statistics of $(q_1,\,q_2)$ by $40$D full dynamics. We use `40D vs 4D' to represent the results of this option.
    \item The third option considers the comparison of progressively coarse-graining. We firstly approximate the full dynamics by either Approach~1 or Approach~2 via option 2.
    Once we get the data for $(q_1,\,\dots, q_4)$ from the approximations, we then compare the statistics of $(q_1,\,q_2)$ by this $4$D surrogate dynamics with those by the surrogate $2$D dynamics in option 1. We use `4D vs 2D' to represent the results of this option.
\end{enumerate}
These scenarios allow us to directly compare the ACF errors for two variables of interest at different levels of fidelity. Also notice that when comparing the performance of Approach 1 and Approach 2, we need to focus on results of `40D vs 2D' and `40D vs 4D'.
    
  We fix $k_1=1$ throughout the simulations, and compare the ACF of the full dynamics, Approach~1 and Approach~2 under the different 
 coarse-graining scenarios where the values of $k_2$ and $k_3$ vary. A comparison of the {relative errors} versus the lag time $\tau\in(0,\,25]$ is plotted for the different parameter choice and approaches in Figure~\ref{Fig:ACF_chain_40}.
We summarise our findings below:
\begin{itemize}
    \item 
    In each setting considered, Approach~2 consistently produces smaller ACF errors than Approach~1.
    \item For Approach~1, the results of progressively coarse-graining agree the prediction in Corollary~\ref{cor:coarsening}. The ACF errors of `40D vs 2D' are always the biggest, which is the case where there is the greatest level of coarse-graining from $\bq$ directly to $\bx$.
    \item For Approach~2, we have seen above that there need not be a strictly monotonic decrease in error for a given time-lag as we increase the level of coarse-graining. However, in the second row of Figure~\ref{Fig:ACF_chain_40}, we do see overall that the ACF error for each 2D model tends to be greater than that for the 4D models. 
   \item We note that the value of $k_2$ appears to determine the approximations of $2$D reduced system and $k_3$ dominates the approximations of $4$D reduced system. This is because the spectral gap $(\lambda-1)$ and deviation angle $\theta$ both increase when $k_2$ becomes bigger for the  `2D' scenario and when $k_3$ becomes bigger for the '4D' scenario. We also notice that the deviations from optimal projections dominate the relative errors for both Approach~1 and Approach~2 due to the fact that the relative spectral gaps are close among different settings of $(k_2,\,k_3)$.
\end{itemize}
Overall, these results provide further evidence that Approach~2 provides a better approximation of dynamical properties of the coarse-grained system than Approach~1, particularly when there exists moderate to large scale separation in the full dynamics.

\begin{figure}[ht!]
    \centering
    \includegraphics[width = 0.95\textwidth, height = 7 cm]{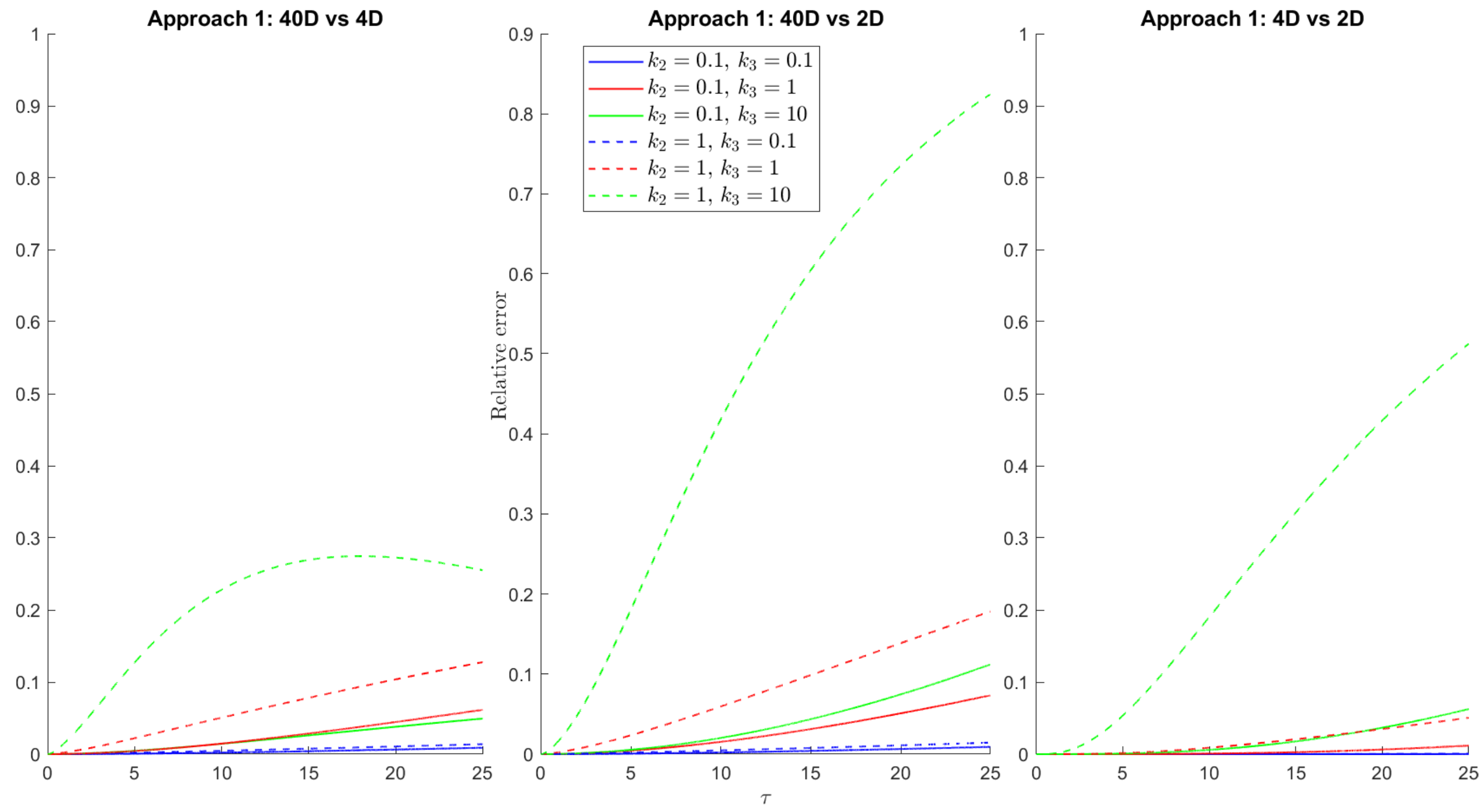}
     \includegraphics[width = 0.95\textwidth, height=7 cm]{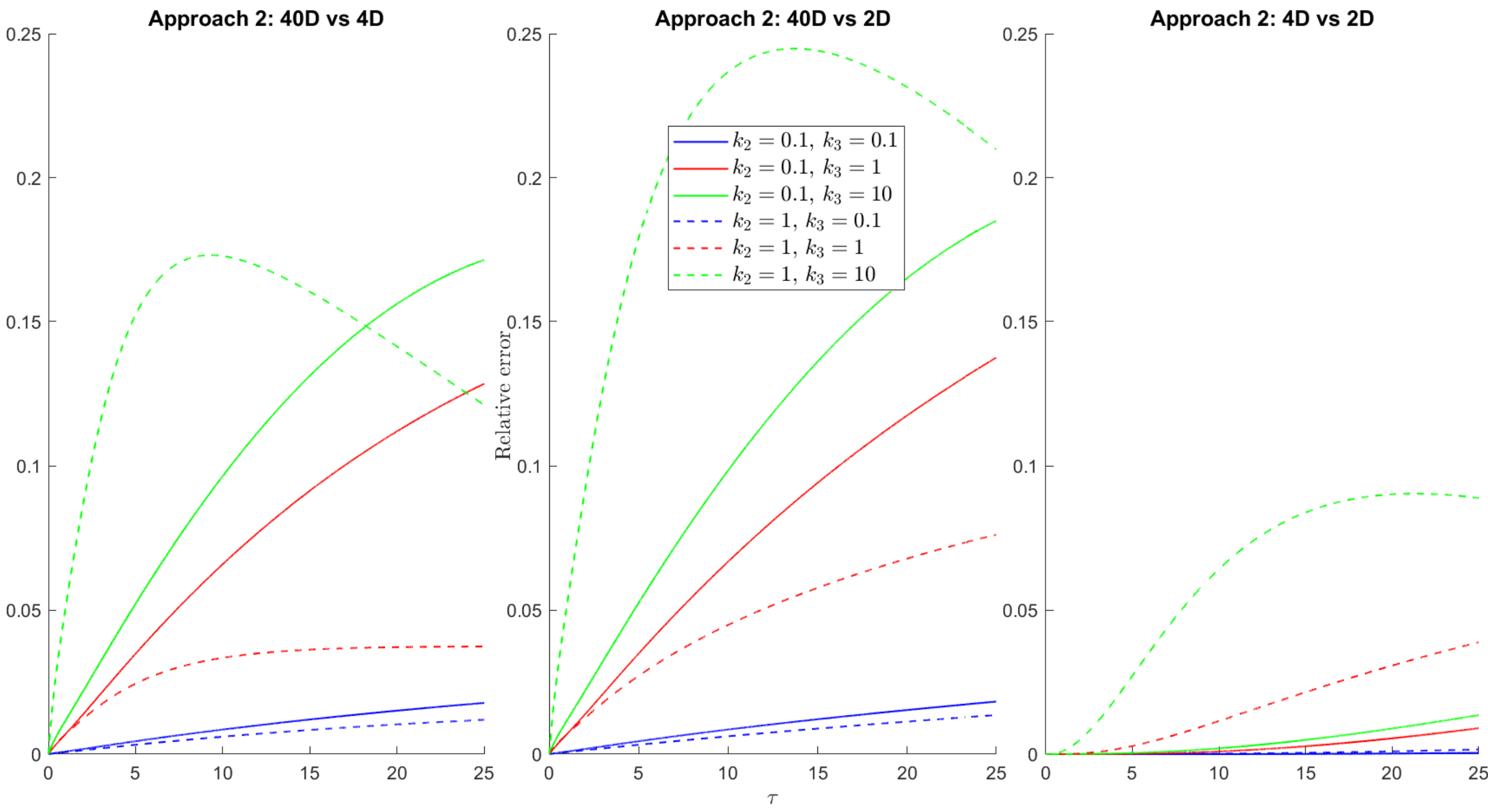}
    \caption{Plot of relative ACF errors for $(q_1,\, q_2)$ as predicted by the full dynamics and different levels of coarsening for Approach~1 and Approach~2. The three coarse-graining scenarios considered, Option 1: `40D vs 2D'; Option 2: `40D vs 4D' and Option 3: `4D vs 2D' are defined in Section~\ref{subsec:1Dchain_num}. Legends indicate the model parameters, which are consistent across each approach. 
    }
    \label{Fig:ACF_chain_40}
\end{figure}

\section{Conclusion}\label{sec:conclusion}
We have derived a framework to construct higher-order Markovian approximation of SDEs by expanding the integral representations of SDEs and hierarchically truncating them at various orders. We then analytically studied the long-time equilibrium properties and the short-time dynamical statistics of proposed approximations in two and higher dimensions. In particular, the dependence of the error in the dynamical autocovariance predicted by our reduced models is quantified in terms of both the timescale separation in the dynamics and the choice of coarse-grained variables, and the asymptotic dependence of the errors were explicitly specified in two-dimensional systems. For high-dimensional systems, we proved that autocovariance errors increase monotonically as the system is progressively coarse-grained using approximation Approach~1. While these features have been observed in numerical simulations in the past, our results provide the first theoretical confirmation that this must always happen.

Our analysis identified the conditions under which autocovariance errors are likely to increase more slowly with system timescale separation; the choice of coarse-grained variables; and the time-scale of simulation interest.

In future work, we aim to extend the current approximation strategy and analysis to other linear systems which act as prototypes for a range of more realistic dynamics for applications in statistical physics such as the underdamped Langevin system. We also plan to study nonlinear systems through a systematic expansion along the low-dimensional structure such as a slow manifold.

\paragraph{Acknowledgements.} X. Li is grateful for partial support by the NSF Award DMS-1847770 and internal Faculty Research Grants. T. Hudson is grateful for support during this collaboration received via a Leverhulme Trust Early Career Fellowship, ECF-2016-526. Both authors are grateful for the generous hospitality of the Institute for Pure and Applied Mathematics at UCLA during part of this work, and for the comments of the anonymous reviewers whose feedback helped to improve this manuscript significantly.

\paragraph{Rights retention statement.}
For the purpose of open access, the authors have applied a Creative Commons Attribution (CC BY) licence to any Author Accepted Manuscript version arising from this submission.

\appendix
\section{Proofs}
\label{app:proofs}

\subsection{Proof of Lemma~\ref{th:EquivalenceLemma}}\label{sec:EquivalenceProof}

    To prove Lemma~\ref{th:EquivalenceLemma}, we note that the equivalence of the first two statements is easy to verify using the definition $\bB = \bA_0-\balpha\bA_1^{-1}\balpha^*$ and the fact that $\bA_1$ is always a positive definite symmetric matrix under our standing assumptions on $\bA$, $\Phi$ and $\Psi$.
    
    To prove equivalence between the second and third statements, we first suppose that $\balpha=\bzer$. Let $\bU$ be the square matrix
    \[
    \bU = \left(\begin{array}{c}\Phi \\ \Psi\end{array}\right),
    \]
    so that
    \[
    \bU\bA\bU^* = \left(\begin{array}{cc}
        \bA_0 & \bzer\\
        \bzer & \bA_1
    \end{array}\right).
    \]
    From this block representation, we see that if $\bv=(v_1,\ldots,v_N)$ is an eigenvector of $\bU\bA\bU$, then so too are the vectors $\bv_0 = (v_1,\ldots,v_n,0\ldots,0)$ and $\bv_1=(0,\ldots,0,v_{n+1},\ldots,v_N)$. In particular, since $\bU\bA\bU^*$ is symmetric and strictly positive definite, there must be $n$ linearly independent eigenvectors of $\bU\bA\bU^*$ with $v_1,\ldots,v_n$ which are not all zero. Moreover, taking a corresponding $\bv_0$ vector, we have
    \[
        \lambda\bv_0 = \bU\bA\bU^*\bv_0
        \quad\Leftrightarrow\quad
        \lambda\bU^*\bv_0 = \bA\bU^*\bv_0
        \quad\Leftrightarrow\quad
        \lambda\Phi^*\left(\begin{array}{c}
             v_1  \\
             \vdots \\
             v_n
        \end{array}\right)
        =\bA\Phi^*\left(\begin{array}{c}
             v_1  \\
             \vdots \\
             v_n
        \end{array}\right).
    \]
    It follows that $\Phi^*\bv_0$ is an eigenvector of $\bA$, and hence there are $n$ distinct eigenvectors of $\bA$ spanned by the rows of $\Phi$.
    
    To prove the other implication, we note that by assumption, there exists a basis of
    $n$ vectors $\bv_0\in\R^n$ such that 
    \[
    \lambda\Phi^*\bv_0=
        \bA\Phi^*\bv_0.
        \]
    As $\Psi\Phi^*=\bzer$ due to the orthogonality of projections. we may apply $\Psi$ on the left to see that
    \[
      \bzer=\Psi\bA\Phi^*\bv_0 = \balpha^*\bv_0.
    \]
    Since $\balpha^*\in\R^{m\times n}$ vanishes on a basis of $\R^n$, it must be that $\balpha=\bzer$, completing the proof.

\subsection{Proof of Theorem~\ref{th:NDError}}\label{sec:NDErrorProof}
In this section, we provide a full proof of Theorem~\ref{th:NDError}. We approach each inequality in turn.\smallskip

    \noindent   
    \textit{Approach 1 lower bound.} We note that $f(x):=x\e^{-\tau/x}$ is a monotone convex function for any $\tau\geq0$, which can be verified directly by differentiating. Applying Theorem~\ref{th:Jensen}, we have that
    \[
    \bR_1(\tau) =  (\Phi\bA^{-1}\Phi^*)\e^{-\tau(\Phi\bA^{-1}\Phi^*)^{-1}}=f\big(\Phi\bA^{-1}\Phi^*\big)\leq \Phi f(\bA^{-1})\Phi^*= \Phi\Big(\bA^{-1}\e^{-\tau\bA}\Big)\Phi^*=\bR(\tau),
    \]
    implying the lower bound.\smallskip
    
    \noindent    
    \textit{Approach 1 upper bound.}
    Consider the function $g(x):=\tfrac12\tau^2/x-x\e^{-\tau/x}$, which is again a monotone convex function for any $\tau\geq0$. Applying Theorem~\ref{th:Jensen}, we have
    \begin{align*}
      \tfrac12t^2\bB-\bR_1(\tau) &= \tfrac12\tau^2(\Phi\bA^{-1}\Phi^*)^{-1}-(\Phi\bA^{-1}\Phi^*)^{-1}\e^{-\tau(\Phi\bA^{-1}\Phi^*)}\\
      &\leq \Phi\Big(\tfrac12\tau^2\bA-\bA^{-1}\e^{-\tau\bA}\Big)\Phi^*\\
      &=\tfrac12\tau^2\bA_0-\bR(\tau),    
    \end{align*}
    which, after rearranging, implies the desired upper bound.\smallskip
    
    \noindent
    \textit{Approach 2 lower bound.}
    To prove Approach 2 bounds, we begin by defining
    \[
    \bOm:=\bC^{\frac12}\Phi\in\R^{n\times N},
    \]
    Now, using the second part of the statement in Theorem~\ref{th:Jensen}, we obtain
    \begin{align*}
        \bR_2(\tau)+\tau\bC &= \bB^{-1}\e^{-t\bB\bC}+\tau\bC\\
        &= 
        \bC^{\frac12}\Big(f(\bOm\bA^{-1}\bOm^*)+\tau\bI\Big)\bC^\frac12\\
        &\leq \bOm \Big(f(\bA^{-1})+\tau\bI\Big)\bOm^*\\
        &= \bC^{\frac12} \Big(\bR(\tau)+\tau\bI\Big)\bC^{\frac12}\\
        &\leq \bR(\tau)+\tau\bI.
    \end{align*}
    Rearranging the final inequality yields the desired result.\smallskip
    
    \noindent
    \textit{Approach 2 upper bound.}
    In this case, we consider
    \begin{align*}
        \tfrac12\tau^2\bC\bB\bC-\tau\bC-\bR_2(\tau)&=
        \bC^{\frac12}\Big(g(\bOm\bA^{-1}\bOm^*)-\tau\bI\Big)\bC^{\frac12}\\
        &\leq\bOm g(\bA^{-1})\bOm^*-\tau\bC\\
        &=\bC^{\frac12}\Big(\tfrac12\tau^2\bA_0-\bR(\tau)-\tau\bI\Big)\bC^{\frac12}\\
        &\leq \tfrac12\tau^2\bA_0-\bR(\tau)-\tau\bI,
    \end{align*}
    and hence
    \[
    \bR(\tau)-\bR_2(\tau)\leq \tfrac12\tau^2(\bA_0-\bC\bB\bC)+\tau(\bC-\bI).
    \]
    Now, since $\bC\leq\bI$ and $\bA_0\geq\bB$, it follows that the operators on the right-hand side of this inequality are negative for small enough $\tau$, and hence there exists $\tau^*>0$ (in particular, taking the value of the first $\tau>0$ for which $0$ is an eigenvalue of the matrix on the right-hand side) such that
    \[
        \bR(\tau)-\bR_2(\tau)\leq 0\quad\text{for all }0\leq \tau\leq \tau^*.
    \]
    This completes the proof of the second upper bound, and hence of the Theorem.

\subsection{Proof of Corollary~\ref{cor:coarsening}}
\label{sec:coarseningproof}

Here, we give a proof of Corollary~\ref{cor:coarsening}. \smallskip

\noindent
\textit{Approach 1 estimates.}
We note that for any fixed $t\geq0$, $f(x):(0,+\infty)\to\R$ defined via $f(x):= x\e^{-t/x}$ is a monotone and convex function. 
    Applying Theorem~\ref{th:Jensen}, we have that
    \begin{align*}
    \bY\bA^{-1}\e^{-t\bA}\bY^*&=\bX\bigg( \bPhi \big(\bA^{-1}\e^{-t\bA}\big)\bPhi^*\bigg)\bX^*\\
    &\geq\bX (\Phi\bA^{-1}\Phi^*)\e^{-t(\Phi\bA^{-1}\Phi^*)^{-1}}\bX^*\\
    &=\bX \Big((\Phi\bA^{-1}\Phi^*)\e^{-t(\Phi\bA^{-1}\Phi^*)^{-1}}\Big)\bX^*\\
    &\geq  \big(\bX\Phi\bA^{-1}\Phi^*\bX^*\big)\e^{-t(\bX\Phi\bA^{-1}\Phi^*\bX^*)^{-1}}\\[2mm]
    &= (\bY\bA^{-1}\bY^*)\e^{-t(\bY\bA^{-1}\bY^*)^{-1}}\geq 0,
    \end{align*}
    which entails that 
    \[
       \bY \Cov(\bq_t,\,\bq_{t+\tau})\bY^*\geq \bX\Cov(\bxi_t,\,\bxi_{t+\tau})\bX^*\geq \Cov(\bx_t,\bx_{t+\tau})\geq 0.
    \]
    Next, we note that we can rearrange this inequality to deduce that
    \[
    \begin{split}
        0&\leq  \bY\Cov(\bq_t,\bq_{t+\tau})\bY^*-\bX\Cov(\bxi_t,\bxi_{t+\tau})\bX^* \leq \bY\Cov(\bq_t,\bq_{t+\tau})\bY^*-\Cov(\bx_t,\bx_{t+\tau}),
        \end{split}
    \]
    and 
    \[
        \bY\Cov(\bq_t,\bq_{t+\tau})\bY-\Cov(\bx_t,\bx_{t+\tau})\geq 
       \bX\Cov(\bxi_t,\bxi_{t+\tau})\bX^*-\Cov(\bx_t,\bx_{t+\tau})\geq 0.
    \]
    Taking the Frobenius norm, which preserves the L\"owner ordering on symmetric positive definite matrices, we obtain the stated estimates.\smallskip

\subsection{Proof of Theorem~\ref{th:2Derror-longtime}}
\label{app:proof-2Derror-longtime}

Here, we provide a proof of the results on the asymptotic error for systems with large time-scale separations in two dimensions.

Using the definitions of $\bA$, $\Phi$ and $\Psi$ in this setting and all relevant submatrices, we note that 
\begin{equation}\label{eq:2D_BandC}
B=\frac{\lambda}{\lambda\cos^2\theta+\sin^2\theta}\quad\text{and}\quad
C=\frac{(\lambda\cos^2\theta+\sin^2\theta)^2}{\lambda^2\cos^2\theta+\sin^2\theta}.
\end{equation} As such, we have that 
        \[
    B^{-1}=\cos^2\theta+\lambda^{-1}\sin^2\theta \quad\text{and}\quad B = \sec^2\theta+O(\lambda^{-1})
        \]
        as $\lambda\to+\infty$. This entails that
        \[
        \beta R_1(\tau) = B^{-1}\e^{-\tau B} = \cos^2\theta\,\e^{-\tau\sec^2\theta}+O(\e^{-\tau\lambda})+O(\lambda^{-1}).
        \]
        Similarly, we find that
        \[
        \beta R(\tau) = \e^{-\tau}\cos^2\theta+\lambda^{-1}\e^{-\tau\lambda}\sin^2\theta,
        \]
        and so for fixed $\tau>0$, the absolute error satisfies 
        \[
            |R(\tau)-R_1(\tau)| = \beta^{-1}\big(\e^{-\tau}-\e^{-\tau\sec^2\theta}\big)\cos^2\theta+O(\lambda^{-1})
        \]
        as $\lambda\to+\infty$. The relative error bound follows upon noting that for fixed $\tau$,
        \[
        \frac{1}{R(\tau)} = \beta \e^{\tau}\sec^2\theta+O(\lambda^{-1}),
        \]
        so multiplying the expansions together, we obtain the desired result.

        Turning to Approach 2, we note that
        \[
        BC = \frac{\lambda\sin^2\theta+\lambda^2\cos^2\theta}{\sin^2\theta+\lambda^2\cos^2\theta}=1+\frac{
    \tan^2\theta}{\lambda}-\frac{
    \tan^2\theta}{\lambda^2}+O(\lambda^{-3}).
        \]
        Substituting this expansion and Taylor expanding, we find
        \begin{align*}
            \beta R_2(\tau) = B^{-1}\e^{-\tau BC} &= \big(\cos^2\theta+\lambda^{-1}\sin^2\theta\big)\exp\Big(-\tau-\frac{\tau}{\lambda}\tan^2\theta+\frac{\tau}{\lambda^2}\tan^2\theta+O(\lambda^{-3})\Big)\\
            &=\big(\cos^2\theta+\lambda^{-1}\sin^2\theta\big)\e^{-\tau}\Big(1-\frac{\tau}{\lambda}\tan^2\theta+O(\lambda^{-2})\Big)\\
            &=\e^{-\tau}\cos^2\theta+\frac{1-\tau}{\lambda}\e^{-\tau}\sin^2\theta+
\tau \e^{-\tau}\sin^2\theta \frac{1 + (\tau-\frac12) \tan^2\theta}{\lambda^2}            
            +O(\lambda^{-3}).
        \end{align*}
        This expansion and those constructed for $R(\tau)$ and $R(\tau)^{-1}$ above immediately yield the result sought.
        The relative error bound at $\tau=1$ follows from the higher-order expansion of $R_2(\tau)$ made above.
        
\subsection{Proof of Theorem~\ref{th:2Derror-shorttime}}
\label{app:proof-2Derror-shorttime}

To prove this result, we use various expressions computed in the proof of Theorem~\ref{th:2Derror-longtime}, given in Section~\ref{app:proof-2Derror-longtime} above.

We note that if $\tau\ll \lambda^{-1}\ll 1$, then Taylor expanding about $\tau=0$ gives
\begin{align*}
\beta R(\tau) &= B^{-1}-\tau+
\tfrac12\big(1+(\lambda-1)\sin^2\theta\big)\tau^2+O(\lambda^2\tau^3)\\
\beta R_1(\tau) &= B^{-1}-\tau +\tfrac12\tau^2 B+O(\tau^3)\\
\beta R_2(\tau) &=B^{-1}-C\tau+O(\tau^2)
\end{align*}
This yields
\begin{align*}
|R(\tau)-R_1(\tau)| &= \tfrac12\tau^2\beta^{-1}\big(1+(\lambda-1)\sin^2\theta-B\big)+O(\lambda^2\tau^3)\\  
&=\tfrac12\tau^2\beta^{-1}(\lambda-1)\sin^2\theta+O(\tau^2)+O(\lambda^2\tau^3)
\end{align*}
and noting that $R(\tau)^{-1}=\beta\sec^2\theta+O(\tau)+O(\lambda\tau^2)$, we also have
\[
\frac{|R_1(\tau)-R(\tau)|}{R(\tau)} = \tfrac12\tau^2(\lambda-1)\tan^2\theta+O(\tau^2)+O(\lambda^3\tau^3).
\]
For Approach~2, we find
\begin{align*}
|R_2(\tau)- R(\tau)| &= \beta^{-1}(1-C)\tau+O(\lambda\tau^2)\\  
&=\beta^{-1}\tau\sin^2\theta + O(\tau/\lambda)+O(\lambda\tau^2),
\end{align*}
and so again using the asymptotic properties of $1/R(\tau)$ in this case, we have
\[
\frac{|R_2(\tau)-R(\tau)|}{R(\tau)} = \tau \tan^2\theta+ O(\tau/\lambda)+O(\lambda\tau^2).
\]
This completes the proof.

\bibliographystyle{unsrt}
\bibliography{references}

\end{document}